\documentclass[11pt]{amsart}
\input epsf
\usepackage{latexsym}
\usepackage{amssymb,amsmath,amsthm,amscd, amssymb,
graphicx, enumerate, verbatim}
\usepackage[all]{xy}

\numberwithin{equation}{section}
\newtheorem{cor}[equation]{Corollary}
\newtheorem{lem}[equation]{Lemma}
\newtheorem{prop}[equation]{Proposition}
\newtheorem{thm}[equation]{Theorem}

\newtheorem{fact}[equation]{Fact}
\newtheorem{claim}[equation]{Claim}

\newtheorem{Example}[equation]{Example}

\newtheorem{remark}[equation]{Remark}
\newenvironment{rmk}{\begin{remark}\rm}{\end{remark}}
\def\co{\colon\thinspace}

\newcommand{\e}{\varepsilon}

\newcommand{\tinycirc}{{_{^\circ}}}

\def\a{\alpha}
\def\G{\Gamma}
\def\g{\gamma}
\def\o{\omega}

\def\d{\partial}

\def\o{\omega}

\def\s{\sigma}

\def\Z{\mathbb{Z}}
\def\S1{\bf S^1}

\newcommand{\cS}{\mathcal S}

\newcommand{\cF}{\mathcal F}

\newcommand{\cpone}{{\mathbb{CP}^1}}
\newcommand{\cptwo}{{\mathbb{CP}^2}}
\newcommand{\cpthree}{{\mathbb{CP}^3}}
\newcommand{\cptwom}{{\mathbb{CP}^{2m}}}

\newcommand{\cpq}{{\mathbb{CP}^{q}}}

\newcommand{\cpqplusone}{{\mathbb{CP}^{q+1}}}

\newcommand{\cH}{{\mathcal H}}
\begin{document}

\abovedisplayskip=6pt plus3pt minus3pt
\belowdisplayskip=6pt plus3pt minus3pt

\title[Codimension two souls and cancellation phenomena]
{\bf Codimension two souls and cancellation phenomena}

\rm

\author{Igor Belegradek\and S{\l}awomir Kwasik \and Reinhard Schultz}

\address{Igor Belegradek\\School of Mathematics\\ Georgia Institute of
Technology\\ Atlanta, GA 30332-0160}\email{ib@math.gatech.edu}
\address{S{\l}awomir Kwasik\\Mathematics Department\\
Tulane University\\6823 St. Charles Ave\\New Orleans, LA 70118}
\email{kwasik@tulane.edu}
\address{Reinhard Schultz\\Department of Mathematics\\ University of California Riverside\\ 900 Big Springs Drive\\ Riverside, CA 92521}
\email{schultz@math.ucr.edu}


\date{}
\begin{abstract}
For each $m\ge 0$ we find an open
$(4m+9)$-dimensional simply connected manifold with
complete nonnegatively curved metrics whose
souls are nondiffeomorphic, homeomorphic, and
have codimension $2$.
We give a diffeomorphism classification
of the pairs $(N,\,\textup{soul})$  when
$N$ is the total space of a nontrivial complex line
bundle over $S^7\times\cptwo$; up to diffeomorphism
there are precisely three such pairs, distinguished by
their nondiffeomorphic souls.
\end{abstract}
\maketitle

\section{Introduction}

In dimension $7$ there are several examples of closed
Riemannian manifolds of nonnegative sectional curvature that are
homeomorphic and nondiffeomorphic. Historically,
the first such example is an exotic $7$-sphere
discovered by Gromoll and Meyer as the biquotient
$Sp(2)/\!/ Sp(1)$~\cite{GroMey}. Other examples include
some homotopy $7$-spheres with metrics of nonnegative sectional curvature constructed
by Grove and Ziller~\cite{GroZil}, and examples found
among Eschenburg spaces and Witten manifolds
by Kreck and Stolz~\cite{KS-wit, KS-mod} (see also~\cite{CEZ}).

Our main result gives the first examples of this kind in dimensions
greater than $7$; for example,
we show that $S^7\times\cptwom$ and $Sp(2)/\!/ Sp(1)\times\cptwom$ are 
not diffeomorphic. This is proved via
a delicate argument which mixes surgery theory
with homotopy-theoretic considerations
from~\cite{Sch73, BecSch, Sch87, Bro-surv}.

Recall that for every integer $d$ there is a unique
oriented homotopy $7$-sphere $\Sigma^{7}(d)$ that bounds a parallelizable
manifold with signature $8d$~\cite{KerMil}.
Here $\Sigma^{7}(0)=S^{7}$, and
$\Sigma^{7}(1)=Sp(2)/\!/ Sp(1)$ generates $bP_{8}\cong\mathbb Z_{28}$,
the group of oriented homotopy $7$-spheres, which all
bound parallelizable manifolds.

Grove and Ziller~\cite{GroZil} constructed metrics of nonnegative sectional curvature
on all exotic $7$-spheres that are the
total spaces of linear $S^3$-bundles over $S^4$.
A classification of such exotic spheres by Eells and Kuiper~\cite{EelKui}
then implies that $\Sigma^7(d)$ admits a metric of nonnegative sectional curvature if
$d\equiv\frac{1}{2}\,h(h-1)\ \text{mod}\ 28$ for some integer $h$.
Since $\Sigma^7(-d)$ and $\Sigma^7(d)$ are orientation-reversingly
diffeomorphic, we see that an oriented homotopy $7$-sphere $\Sigma^7(d)$
admits a metric
of nonnegative sectional curvature provided $d\notin\{2,5,9,12\}$.

One obvious approach to constructing additional manifolds with
metrics of nonnegative sectional curvature is to take products of
previously known examples.  In particular, products of a fixed manifold of
nonnegative sectional curvature and the exotic
spheres $\Sigma^7(d)$, with $d$ restricted as above, are candidates for
pairs of homeomorphic but not diffeomorphic manifolds supporting
such metrics.  However, this often does not yield new diffeomorphism types;
for example, it is well known that $\Sigma^7(d)\times S^k$ is diffeomorphic to
$S^7\times S^k$ for all $k\geq 2$ ({\it e.g.\/}, see \cite{Sc1}).  We shall
prove the following result, which
does yield new families of homeomorphic but nondiffeomorphic pairs with
metrics of nonnegative sectional curvature:

\begin{thm}\label{intro-thm: CPq cancells}
If $m$, $d$, $d^\prime$ are integers and
$d-d^\prime$ is odd, then
$\Sigma^7(d)\times \cptwom$ is not diffeomorphic to
$\Sigma^7(d^\prime)\times \cptwom$.
\end{thm}

The proof of Theorem~\ref{intro-thm: CPq cancells} occupies
most of the paper and is sketched in Section~\ref{sec: sketch}.
As we see in Lemma~\ref{lem: inertia vs divisible by 3} below,
$\Sigma^7(d)\times \cptwom$ and
$\Sigma^7(d^\prime)\times \cptwom$ are diffeomorphic if
either $d-d^\prime$ or $d+d^\prime$ is
divisible by $4$, and $m$ is not divisible by $3$.
For $m=1$ we have the following optimal strengthening of
Theorem~\ref{intro-thm: CPq cancells}:

\begin{thm}\label{intro-thm: CP2 cancells}
If $d$, $d^\prime$ are integers, then
$\Sigma^7(d)\times \cptwo$ is diffeomorphic to
$\Sigma^7(d^\prime)\times \cptwo$ if and only if
either $d-d^\prime$ or $d+d^\prime$ is divisible by $4$.
\end{thm}

Therefore, if $m$ is not divisible by $3$, then
each $\Sigma^7(d)\times \cptwom$ admits a metric
of nonnegative sectional curvature, and the manifolds $\Sigma^7(d)\times \cptwom$
lie in $2$ or $3$ unoriented diffeomorphism classes;
for $m=1$ they lie in $3$ unoriented diffeomorphism classes.

We shall also show that every manifold that is tangentially homotopy
equivalent to $\Sigma^7(d)\times \cptwo$ is diffeomorphic to
$\Sigma^7(d^\prime)\times \cptwo$ for some $d^\prime$
(see Section~\ref{sec: tang hom eq}).

The problem of determining whether products of $\mathbb {CP}^n$ with homotopy spheres
are diffeomorphic goes back to Sullivan's results on product formulas for surgery
obstructions, and in \cite{Bro-surv} Browder showed the
relevance of such results to constructing smooth semifree circle actions
on homotopy $(2k+7)$-spheres. In particular, by~\cite[Section 6]{Bro-surv}
the specialization of
Theorem~\ref{intro-thm: CPq cancells}
to $d=1$ immediately yields the following result:

\begin{cor}\label{intro-cor: circle action}
Given an odd integer $d$ and positive integer $k$,
the exotic sphere $\Sigma^7(d)$
is diffeomorphic to the fixed point set of a smooth semifree
circle action on a homotopy $(2k+7)$-sphere if and only
if $k$ is even.
\end{cor}

The result is new for $k=3$, and the cases $k\geq 5$ are stated
in a paper by the third author (compare~\cite[Theorem III]{Sch87})
with a derivation based upon some results whose proofs have not yet
been published. The corresponding result when $k=1$ follows from a much
older result of W.-Y. Hsiang~\cite[Theorem II]{Hsi}.
On the other hand, if $k$ is even, then~\cite[Theorem 6.1]{Bro-surv} implies
that every product $\Sigma^{7}(d)\times\mathbb {CP}^{k-1}$ is
diffeomorphic to $S^{7}\times\mathbb {CP}^{k-1}$, and that
every homotopy $7$-sphere can be realized as a fixed point
set of a smooth semifree $S^1$-action on a homotopy $(2k+7)$-sphere.

We shall apply the preceding results to study manifolds of nonnegative
sectional curvature via the {\sl soul theorem} of Cheeger and Gromoll.
Their results imply that every open complete manifold $N$ of nonnegative sectional curvature
is diffeomorphic to the total space of a normal bundle to
a compact totally geodesic submanifold, called a {\sl soul}~\cite{CheGro}.
The diffeomorphism class of a soul may depend on the
metric, and this dependence has been investigated
in~\cite{Bel, KPT} and more recently in~\cite{BKS1},
where the reader can find further motivation and background.
In particular, in~\cite{BKS1} we systematically searched
for open manifolds admitting metrics with nondiffeomorphic
souls of lowest possible codimension. To this end we show the following:

\begin{thm}\label{intro-thm: nondiffe souls}
For each $m\ge 0$ there exists an open
$(4m+9)$-dimensional simply connected manifold $N$ admitting
complete metrics of nonnegative sectional curvature whose
souls are nondiffeomorphic, homeomorphic, and
have codimension $2$.
\end{thm}

If $\dim(N)\ge 6$ then codimension $2$ is the smallest possibility, for
the $h$-cobordism theorem implies that all codimension $1$ souls
in a simply connected open manifold of dimension at least $6$
are diffeomorphic.
Taking $N$ in Theorem~\ref{intro-thm: nondiffe souls} to
be the product of $\mathbb R^2$ with manifolds in
Theorem~\ref{intro-thm: CPq cancells} cannot work
because a closed simply connected manifold of dimension at least 5
can be recovered up to diffeomorphism from its product
with $\mathbb R^2$. Instead, we find nontrivial $\mathbb R^2$-bundles
over manifolds in Theorem~\ref{intro-thm: CPq cancells}
that admit metrics of nonnegative sectional curvature and have diffeomorphic total
spaces. The same reasoning
works for $\mathbb R^2$-bundles over Eschenburg
spaces or Witten manifolds
that are homeomorphic and nondiffeomorphic, which covers
the case $m=0$ in Theorem~\ref{intro-thm: nondiffe souls}.

Let $\mathfrak M^{k,c}_{\sec\ge 0} (N)$
denote the moduli space of complete metrics of nonnegative sectional curvature on $N$
with topology of $C^k$-convergence on compact
subsets, where $0\le k\le \infty$.  If $N$
admits a complete metric with nonnegative sectional curvature
whose soul has non-trivial normal Euler class,  then
the results of~\cite{KPT} show that metrics with nondiffeomorphic
souls lie in different components of
$\mathfrak M^{k,c}_{\sec\ge 0} (N)$.  More generally,
the authors showed in~\cite{BKS1} that associating to
the nonnegatively curved metric $g$ the diffeomorphism type
of the pair $(N,\,\textup{soul of}\ g)$ defines a locally constant
function on $\mathfrak M^{k,c}_{\sec\ge 0} (N)$;
thus nondiffeomorphic pairs correspond to
metrics in different components of the moduli space.
Since the souls in Theorem~\ref{intro-thm: nondiffe souls}
have nontrivial normal Euler class, we obtain the following:

\begin{cor}\label{intro-cor: mod}
$\mathfrak M^{k,c}_{\sec\ge 0} (N)$ is not connected
for $N$ as in Theorem~\textup{\ref{intro-thm: nondiffe souls}}.
\end{cor}

Given an open manifold $N$ admitting a complete metric of nonnegative sectional curvature
with soul $S_0$,
an attractive goal is to obtain a diffeomorphism classification
of pairs $(N, S)$ where $S$ is a soul of some complete metric
of nonnegative sectional curvature on $N$. Here we focus on the case when
$S_0$ is a simply connected, has codimension $2$,
and dimension at least 5. If $S_0$ has trivial normal bundle,
and $S$ is any other soul in $N$, then the pairs
$(N, S)$ and $(N, S_0)$ are diffeomorphic~\cite[Lemma 5.8]{BKS1}.
To our knowledge the results below
are the first instances of diffeomorphism classification of the
pairs $(N,S)$, where $S$ is a soul of some complete metric
of nonnegative sectional curvature on $N$, in which not all such pairs are diffeomorphic.

\begin{thm}\label{intro-thm: classification for s7xcp2}
The total space $N$ of every nontrivial complex line bundle
over $\cptwo\times S^7$ admits $3$ complete nonnegatively
curved metrics with pairwise nondiffeomorphic souls
$S_0$, $S_1$, $S_2$ such that for every complete nonnegatively
curved metric on $N$ with soul $S$, there exists a self-diffeomorphism
of $N$ taking $S$ to some $S_i$.
\end{thm}

Here $S_i$ is isometric to the product $\cptwo\times\Sigma^7(3i)$
where the second factor is given a metric of nonnegative sectional curvature
by~\cite{GroZil}.

One can prove similar result for $7$-dimensional
souls which are certain Witten manifolds.
Recall that the {\sl Witten manifold} $M_{k,l}$
is the total space of an oriented circle bundle
over $\cptwo\times\cpone$ with Euler class given by
$(l,k)\in H^2(\cptwo )\oplus H^2(\cpone )$ where $l$,
$k$ are nonzero coprime integers.
In~\cite[Theorem B]{KS-wit} Kreck and Stolz  classified the
Witten manifolds $M_{k,l}$ up to oriented
homeomorphism and diffeomorphism in terms of $k,l$,
and the above definition of $M_{k,l}$
easily implies that $M_{-k,-l}$ is orientation-reversingly
diffeomorphic to $M_{k,l}$, so one
also has a (unoriented) diffeomorphism classification
of Witten manifolds.
As remarked after~\cite[Corollary C]{KS-wit},
if $l\equiv 0\ \mathrm{mod}\, 4$
and $l\equiv 0, 3, 4\ \mathrm{mod}\, 7$,
then every smooth manifold that is homeomorphic to
$M_{k,l}$ must be a Witten manifold.
For these examples we shall prove he following result:

\begin{thm}
\label{intro-thm: classification for witten}
For nonzero coprime integers $k,l$ with
$l\equiv 0, 3, 4\, \mathrm{mod}\, 7$ and
$l\equiv 0\, \mathrm{mod}\, 4$,
let $N$ be the total space of a nontrivial rank $2$ vector bundle
over $M_{k,l}$. \newline
\textup{(1)}
If the Witten manifold $M_{k^\prime,\, l^\prime}$ is
homeomorphic to $M_{k,l}$, then $N$ has a complete metric
of nonnegative sectional curvature whose soul $S_{k^\prime,l^\prime}$ is
diffeomorphic to $M_{k^\prime,l^\prime}$. \newline
\textup{(2)}
For every complete metric
of nonnegative sectional curvature on $N$ with soul $S$ the pair $(N,S)$
is diffeomorphic to $(N, S_{k^\prime,\, l^\prime})$
for some $S_{k^\prime, l^\prime}$ as in \textup{(1)}.
\end{thm}

Other examples of homeomorphic, nondiffeomorphic,
closed $7$-manifolds of nonnegative sectional curvature occur among the so-called
Eschenburg spaces, which is a family
of quotients of $SU(3)$ by free circle actions~\cite{CEZ}.

The proof of Theorems~\ref{intro-thm: classification for s7xcp2}
and~\ref{intro-thm: classification for witten} hinges on the
following three observations.

$\bullet$ If $S$, $S^\prime$ are homeomorphic, nondiffeomorphic
manifolds that are Eschenburg spaces, or Witten manifolds, or products
$\Sigma^7(d)\times\cptwom$, then $S^\prime$ is the connected sum
of $S$ with a homotopy sphere that bounds a parallelizable
manifold.

$\bullet$ If $S^\prime$ is a closed simply connected manifold
of dimension at least 5, and if $S^\prime$ is diffeomorphic
to the connected sum of $S$ and a homotopy sphere
that bounds a parallelizable manifold, then
there are nontrivial $\mathbb R^2$-bundles over $S$ and $S^\prime$
with diffeomorphic total spaces
(see Theorem~\ref{thm: diffeo total spaces}).

$\bullet$ If $S$ and $S^\prime$ are simply connected souls of codimension $2$
and dimension at least 5, then $S$ is diffeomorphic
to the connected sum of $S^\prime$ and a homotopy sphere
that bounds a parallelizable manifold
(see~\cite[Theorem 1.9]{BKS1} where the conclusion
that the homotopy sphere bounds a parallelizable manifold
is missing, but is implied by the proof there).

The structure of the paper is as follows.
Section~\ref{sec: surgery basics} supplies the needed
background on surgery theory, and
Section~\ref{sec: sketch} contains a sketch of the proof of
Theorem~\ref{intro-thm: CPq cancells}.
In Section~\ref{sec: dihotomy} we explain how
Theorem~\ref{intro-thm: CPq cancells} follows from
what we call the {\bf Dichotomy Principle}:
A homotopy self-equivalence of
$S^7\times\cptwom$ is either homotopic to
a diffeomorphism or else has a nontrivial normal invariant.
As a starting point in proving the Dichotomy Principle,
in Section~\ref{sec: factorization} we describe a
canonical factorization of a homotopy self-equivalence $f$ of
$S^7\times\cptwom$ into the
composition of a diffeomorphism and two
homotopy self-equivalences each arising from a map
of one factor of $S^7\times\cptwom$
into the space of homotopy self-equivalences
of the other factor.
Proving the Dichotomy Principle
for homotopy self-equivalences coming from
maps of $S^7$ into the space of homotopy self-equivalences
of $\cpq$ will require $(i)$ a spectral sequence from
\cite{Sch73} which is studied further in
Section~\ref{sec: spectral seq} and $(ii)$ properties of
surgery-theoretic structure sets which are described in
Section~\ref{sec: self and structure sets}.
Sections~\ref{sec: stably trivial is trivial}--\ref{sec: order 2 nontrivial normal inv}
analyze how certain groups of homotopy classes map into surgery-theoretic
structure sets.  The Dichotomy Principle is finally established in
Section~\ref{sec: dichotomy proof and skeleton filtrations}, and
Theorem~\ref{intro-thm: CP2 cancells} is proved in
Section~\ref{sec: index 4}.
As a by-product of our methods, we give in
Section~\ref{sec: tang hom eq}  a diffeomorphism
classification of manifolds that are
tangentially homotopy equivalent to $S^7\times\cptwo$.
Section~\ref{sec: codim 2 souls} contains a surgery theoretic argument
which proves Theorem~\ref{intro-thm: nondiffe souls}.

{\bf Acknowledgments. }
The first author (I. B.)  was partially supported by NSF Grant \# DMS-0804038,
and the second author (S. K.) was partially supported by State of Louisiana Board of
Regents Award LEQSF(2008-2011)-RD-A-24.
We are especially grateful to the referee for an
extraordinarily thorough report on
this paper and a wide range of very constructive comments.

\section{On classifying smooth manifolds via surgery}
\label{sec: surgery basics}

In this section we describe some results of surgery theory
that are used throughout this paper.
Background references for surgery are Wall's book~\cite{Wal-book},
especially Chapters $3$ and $10$,
Browder's book~\cite{Bro-book} for the simply connected case,
and the more recent book by Ranicki~\cite{Ran-book}.

Let $M^n$ be a compact smooth manifold, with or
without boundary, where
both $M$ and $\partial M$ are assumed connected unless stated otherwise.
We also assume that $n\ge 6$ if $\d M\neq\emptyset$,
and $n\ge 5$ otherwise.
A {\sl simple homotopy structure on} $M$ is a
pair $(N,f)$ consisting of a compact smooth manifold $N$ and
a simple homotopy equivalence of manifolds with boundary
(in other words, a homotopy equivalence of pairs).
Two such structures $(N_1,f_1)$ and $(N_2,f_2)$ are said to
be {\sl equivalent} if there is a diffeomorphism
$h:N_1\to N_2$ such that ${f_2}\tinycirc h\simeq f_1$,
where again the homotopy is a homotopy of pairs.  The set
of all such equivalence classes is a pointed set which
is often called  the {\it simple structure set} of $M$ and
denoted by ${\bf S}^s(M)$.
Its base point is the class of the identity $(M,{\bf id}_M)$,
and the pointed set fits into an exact {\sl Sullivan-Wall
surgery exact sequence}
\[
\dots\to
[\Sigma(M/\d M), F/O]\stackrel{\s}{\to}
L^s_{n+1}(\pi_1(M),\pi_1(\partial M))
\stackrel{\Delta}{\to}{\bf S}^s(M)
\stackrel{{\mathfrak q}}{\to} [M,F/O]
\stackrel{\s}{\to}\dots
\]
which continues indefinitely to the left. The space $F/O$
(denoted by $G/O$ in \cite[p. 46]{Bro-book} and many
other references) represents a homotopy functor describable as follows:
Given a compact space $X$, a class in $[X,F/O]$ is an appropriately defined
equivalence class of pairs $(\alpha,\Phi)$ consisting
of a stable vector bundle $\alpha$ over $X$ together with a
stable fiber homotopy trivialization $\Phi$ of its unit sphere
bundle (or equivalently of its fiberwise one point compactification provided
$X$ is compact);
more will be said about $F/O$ at the end of this section.

The surgery exact sequence also continues
one step to the right with a {\sl surgery obstruction map\/}
$\s\co [M,F/O]\to
L^s_{n}\left(\,\pi_1(M),\pi_1(\partial M)\,\right)$, where
the latter is an abelian group,
called the {\sl Wall\/} (or {\sl surgery obstruction\/})
{\sl group\/}.  This
group depends only on the (inclusion induced)
homomorphism $\pi_1(\partial M)\to\pi_1(M)$,
the residue class of $m$ modulo $4$, and
the (orientation) homomorphism $w\co\pi_1(M)\to\mathbb Z_2$,
which we omit from the notation because {\sl we only work with
orientable manifolds in this paper\/}. The map
$\Delta$ comes from an action of $L^s_{n+1}(\pi_1(M),\pi_1
(\partial M))$ on the pointed set ${\bf S}^s(M)$.
The map ${\mathfrak q}$ from ${\bf S}^s(M)$ to the set of
homotopy classes $[M,F/O]$ is called the {\sl normal invariant\/}.

Exactness at the term ${\bf S}^s(M)$ means that
two simple homotopy structures have equal normal
invariants if and only if they are in the same orbit of the group
action which defines $\Delta$. Although $[M, F/O]$ is an abelian
group, the surgery obstruction map $\sigma$
is not necessarily
a homomorphism, and exactness at $[M, F/O]$ means that
the inverse image of its zero element under ${\mathfrak q}$
is equal to $\s^{-1}(0)$;
however, $\s$ becomes a
homomorphism in the continuation
of the surgery sequence to the left
starting with
$[\Sigma(M/\d M), F/O]\to
L^s_{n+1}(\pi_1(M),\pi_1(\partial M))$~\cite[Proposition 10.7]{Wal-book}.

Although surgery theory in principle yields a diffeomorphism
classification for closed manifolds with a fixed homotopy type, it
does so indirectly in terms of homotopy theory, and
a complete classification is known for only a few
homotopy types.

If the inclusion $\d M\to M$ induces an isomorphism of
fundamental groups, then the
relative Wall groups vanish by Wall's $\pi-\pi$ Theorem
\cite[Theorem 3.3]{Wal-book}, and ${\bf S}^s(M)$
maps bijectively to $[M,F/O]$ via $\mathfrak q$.

If $M=S^n$, then ${\bf S}^s(S^n)=\Theta_n$,
the set of oriented diffeomorphism classes of homotopy
$n$-spheres \cite{KM}, which has a group structure defined
by connected sum.
The subgroup $bP_{n+1}$ of homotopy $n$-spheres that bound
parallelizable manifolds can be identified with
the image of the homomorphism $L^s_{n+1}(1)\to {\bf S}^s(S^n)$.

More generally, if $M$ is closed and simply connected then
the action of $L^s_{n+1}(1)$ on ${\bf S}^s(M)$ factors through
the $bP_{n+1}$-action via connected sum, and
for every two homotopy equivalences $f_1, f_2\co N\to M$
with equal normal invariants,
$f_1$ is the connected sum of $f_2$ with an
orientation-preserving homeomorphism $\Sigma^n\to S^n$
where $\Sigma^n$ represents an element of $bP_{n+1}$;
in particular, $N_1$ is diffeomorphic to $N_2\,\#\,\Sigma^n$.

The results of~\cite{KerMil}
show that $bP_{n+1}$ is a finite cyclic group
which vanishes if $n$ is even, and has order
at most $2$ if $n=4r+1$ (the order is 1 if $r=0,1,3,7$ or 15, not yet
known if $r=31$, and 2 otherwise). On the other hand, the order of
$bP_{4r}$ grows exponentially with $r$, and this is the case we
study in the present paper.
Each element of $bP_{4r}$ with $r\ge 2$ is represented by a homotopy
sphere $\Sigma^{4r-1}(d)$ that bounds a parallelizable manifold $W$
of signature $8d$; the oriented diffeomorphism type of $\Sigma^{4r-1}(d)$
depends only on the signature of the cobounding manifold $W$, so that
the homotopy sphere $\Sigma^{4r-1}(1)$ generates $bP_{4r}$ and
$\Sigma^{4r-1}(0)=S^{4r-1}$.

We shall illustrate how surgery works for an example which is
central to this paper.
Removing an open disk from the interior of
a parallelizable manifold with boundary
$\Sigma^{4r-1}(d)$ yields a parallelizable
cobordism $W^{4r}$ between $\Sigma^{4r-1}(d)$ and $S^{4r-1}$,
and hence defines a normal map
\[
F\co (W^{4r},\d W^{4r})\to (S^{4r-1}\times I, S^{4r-1}\times \d I)
\]
covered by an isomorphism of trivial stable normal bundles.
The surgery obstruction is preserved by
products with $\cptwom$~\cite[Theorem III.5.4]{Bro-book},
so that $\s(F\times {\bf id}(\cptwom))=\s(F)=d\in L_{4m+4r}(1)=\Z$.
Let $f\co U^{4m+4r}(d)\to D^{4m+4r}$ be a (boundary preserving)
degree one map, where $U^{4m+4r}(d)$ is
a parallelizable manifold that bounds $\Sigma^{4m+4r-1}(d)$.
Taking boundary connected sums of $F\times {\bf id}(\cptwom)$
and $f$ along the boundary component $S^{4r-1}\times\cptwom$
defines a normal map with zero surgery obstruction, hence
it can be turned into a simple homotopy equivalence
via surgery, in other words, we obtain an $s$-cobordism between
$\Sigma^{4r-1}(d)\times\cptwom$ and
$(S^{4r-1}\times\cptwom)\,\#\,\Sigma^{4m+4r-1}(d)$, which are
therefore diffeomorphic. In summary, the following holds:

\begin{fact} \label{fact: browder}
Let $M=S^{4r-1}\times\cptwom$ where $r\ge 2$.
If $h\co\Sigma^{4r-1}(d)\to S^{4r-1}$ and
$H\co \Sigma^{4m+4r-1}(d)\to S^{4m+4r-1}$
are orientation-preserving
homeomorphisms, then the simple homotopy structures
\begin{align*}
h\times{\bf id}(\cptwom)\co \Sigma^{4r-1}(d)\times
\cptwom\to S^{4r-1}\times\cptwom \\
H\,\#\,{\bf id}(M)\co
\Sigma^{4m+4r-1}(d)\,\#\,M\to
S^{4m+4r-1}\,\#\,M=M
\end{align*}
represent the same element $\Delta(d)$
in the structure set ${\bf S}^s(M)$.
\end{fact}

Determining the kernel of $\Delta$
is a major step in the diffeomorphism classification
of closed manifolds homotopy equivalent to $M$.
The {\it homotopy inertia group} $I_h(M)$ of an $n$-manifold $M$
is the group of all $\Sigma\in\Theta_n$ such that
the standard homeomorphism $M\#\Sigma\to M$ is homotopic
to a diffeomorphism. The kernel of
$\Delta\co L_{n+1}^s(\pi_1(M))\to {\bf S}^s(M)$ is called the
{\it surgery inertia group} and denoted $I_\Delta(M)$.
If $M$ is closed and simply connected, then
$I_\Delta(M)$ is the preimage of $I_h(M)\cap bP_{n+1}$
under $\Delta$.

In particular, if $M$ is a closed simply connected $(4r-1)$-manifold,
then $d\in\mathbb Z=L_{4r}^s(1)$ acts on ${\bf S}^s(M)$ by
taking connected sum with $\Sigma^{4r-1}(d)$, and
$\Delta (d)$ is trivial in ${\bf S}^s(M)$
if and only if $\Sigma(d)\in I_h(M)$ (compare~\cite[II.4.10]{Bro-book}).

A key ingredient of our work is the proof
given in~\cite[Theorem 2.1]{Sch87}
of the following result due to L. Taylor \cite{Tay}.

\begin{thm}
\label{thm: taylor}
If $M$ is a closed oriented smooth manifold of
dimension $4r-1\geq 7$, then the subgroup
$I_h(M)\cap bP_{4r}$ of $bP_{4r}$ has index $\ge 2$.
\end{thm}

Taylor's theorem gives the best general estimate for
$I_h(M)\cap bP_{4r}$; {\it e.g.\/}, if
$M=S^3\times\mathbb CP^{2m}$ with $m\geq 1$, then
the index of $I_h(M)\cap bP_{4r}$ in $bP_{4r}$ is $2$
\cite[Example 2, p. 190]{Sch87}.

On the other hand, much sharper estimates exist if $M$ satisfies
some relatively mild restrictions; for example \cite[Theorem 2.13]{Bro-inert}
implies that $I_h(M)\cap bP_{4r}$ is trivial if $M$ is a
simply connected, stably parallelizable closed Spin manifold
of dimension $4r-1\geq 7$.  In Section~\ref{sec: index 4}
we show that $I_h(M)\cap bP_{4r}$ has index $4$ in $bP_{4r}$
if $M=S^7\times\cptwo$.

Even though Taylor's result ensures that
the standard homeomorphism from
$\Sigma^7(1)\,\#\,(S^3\times\cptwo )$ to $S^3\times\cptwo$
is not homotopic to a diffeomorphism,
these two manifolds are diffeomorphic
as proved in~\cite[Corollary 4.2]{MasSch}.

This naturally brings us to another source of nontrivial
elements in ${\bf S}^s(M)$; namely, simple homotopy self-equivalences
of $M$ that are not homotopic to diffeomorphisms.
For the purposes of diffeomorphism classification, we should identify
two simple homotopy structures $f_1, f_2\co N\to M$ which differ
by a simple homotopy self-equivalence of $M$; {\it i.e.\/},
we need to take the quotient of ${\bf S}^s(M)$ by the action of
the group ${\mathcal E}^s(M,\partial M)$ of homotopy classes of simple homotopy
self-equivalences of $(M,\partial M)$ via composition:
\[
[h]\cdot [N,f]~~=~~[N,h\tinycirc f]
\]
where $(N,f)$ represents a class in ${\bf S}^s(M)$ and $h\in
{\mathcal E}^s(M,\partial M)$.
With rare exceptions the group ${\mathcal E}^s(M,\partial M)$ is
extremely hard to compute, even when $M$ is simply connected and
has relatively few nontrivial homology groups;
in this case all homotopy equivalences are simple
so in agreement with earlier notation
we write ${\mathcal E}$ instead of ${\mathcal E}^s$ when
$M$ is simply connected.

When comparing elements of ${\bf S}^s(M)$ that differ by a homotopy
self-equivalence the following
{\it composition formula for normal invariants} is useful
(see~\cite[p. 144]{Sch71} or~\cite[Corollary 2.6]{MTW}):
\[
{\mathfrak q}(g\tinycirc h)~~=~~
{\mathfrak q}(g)~+~\left(g^*\right)^{-1}\,{\mathfrak q}(h)
\]
Here $g$ represents a class in ${\bf S}^s(M)$, while $h$
is a homotopy self-equivalence of $M$ and
the operation ``+'' refers to the abelian group structure in $[M, F/O]$
induced by the Whitney sum in $F/O$.

A more general version of surgery theory includes
{\sl relative simple structure sets\/}
${\bf S}^{s}(M~\text{rel}~\d M)$
of simple homotopy structures $(N,f)$, where
$f:(N,\partial N)\to (M,\partial M)$ is a simple homotopy equivalence
of pairs which is a diffeomorphism on a neighborhood of
$\d N$.  Two such structures
$(N_1,f_1)$ and $(N_2,f_2)$ are said to
be equivalent if there is a diffeomorphism
$h:N_1\to N_2$ such that ${f_2}\tinycirc h$ and $f_1$
are homotopic through a map of pairs that induces a
diffeomorphism on a neighborhood of $\partial N_1\times[0,1]$.
There is a corresponding surgery sequence which is exact
for $n\ge 5$:
\[
[\Sigma(M/\d M),F/O]~\to~L^s_{n+1}\left(\,\pi_1(M)\,\right)~\to~
{\bf S}^{s}(M~\text{rel}~\d M)~\to~[M/\d M,F/O].
\]

We shall need the following proposition that
is a direct consequence of the bordism construction for the
surgery exact sequence in \cite[Chapter 10]{Wal-book}.

\begin{prop}
\label{prop: restriction of normal invariants}
Let $M$ be a closed manifold, and
$M_1$, $M_2$ be codimension zero compact
submanifolds of $M$ such that
$\partial M_1=M_1\cap M_2=\partial M_2$.
Let $h\co (N,\partial N)\to
(M_1,\partial M_1)$ be a simple homotopy structure which is an
isomorphism on the boundary, and let
$H\co N\cup_{\partial h} M_2\to M=M_1\cup M_2$ be
the simple homotopy structure obtained by attaching copies of
$\partial M_2$ along the boundaries.
Then the normal invariants of $H$
and $h$, viewed as elements of $[M,F/O]$ and $[M_1/\partial
M_1,F/O]$, are related by the equation
${\mathfrak q}(H)=c^*{\mathfrak q}(h)$, where
$c:M\to M/M_2\cong M_1/\partial M_1$
is the collapsing map.
\end{prop}

Finally, we summarize some results on classifying
spaces of surgery theory
(see~\cite[Chapter 3A]{MM} and~\cite[Section 9.2]{Ran-book}).
Denote the topological monoid of homotopy
self-equivalences of $S^k$ by $G_{k+1}$,
and denote its identity component by $SG_{k+1}$.
Let $F_k$ denote the submonoid of
$G_{k+1}$ consisting of base point
preserving self-maps, and let $SF_k=F_k\cap SG_{k+1}$.
The evaluation map defines a fibration $G_{k+1}\to S^k$
with fiber $F_k$, which restricts to a fibration $SG_{k+1}\to S^k$
with fiber $SF_k$.

The space $SF_k$ has the homotopy type of the component of the
constant map in the iterated loop space $\Omega^kS^k$, and
therefore $\pi_n(SF_k)\cong\pi_{n+k}(S^k)$ for all $n\geq 1$.
There is a sequence of inclusions
\begin{equation}
\label{form: mixed monoids}
\dots~\to~F_{k-1}~\to~G_k~\to~F_k~\to~G_{k+1}~\to~\dots
\end{equation}
where each map is an injective continuous monoid homomorphism.
The direct limits of the subsequences $\{F_k\}$ and $\{G_k\}$ are
usually denoted $F$ and $G$; this notation is actually redundant,
for $F$ and $G$ are isomorphic as topological monoids
because $\{F_k\}$ and $\{G_k\}$ are cofinal in (\ref{form: mixed monoids}).
In what follows we shall use $F$ in conformity with \cite{BecSch}.
By \cite{Stasheff63} the classifying space $BF$ of $F$ is also a
classifying space for
stable fiber homotopy equivalence classes of spherical
fibrations.

If $O$ denotes the increasing union of the orthogonal groups
$\cup_{n\geq 1} O_n$, then
there exists a (homotopy) exact sequence of $H$-spaces
\begin{equation}
\label{eq: fibration F/O}
O~\to~F~\to~F/O~\to~BO~\to~BF
\end{equation}
in which any three consecutive terms
form a fibration and $F/O$ is the space which previously appeared in the
surgery exact sequence.
If $X$ is a space and we apply the covariant homotopy class functor
$A~~\rightsquigarrow~~[X,A]$
to this fibration sequence, we obtain an
exact sequence of abelian groups
$$\cdots~\to~[\Sigma X,F/O]~\to~[X,O]~\to~[X,F]~\to~[X,F/O]~\to~[X,BO]~\to~[X,BF]$$
where the groups $[X, F]$ and $[X,BF]$ are finite if
$X$ is a finite complex.
We shall also use standard facts on
homotopy groups of $O$, $F$ and $F/O$, and the image of
the $J$-homomorphism $O\to F$, all of which can be found
in~\cite{Ada-J-IV, Hus, Levine83, Ran-book, Tod}.

\section{Sketch of the proof of Theorem~\ref{intro-thm: CPq cancells}}
\label{sec: sketch}

Let $M=S^k\times\cptwom$ where $m\geq 1$ and $k=4r-1\ge 7$.
By Fact~\ref{fact: browder} we know that
$\Sigma^k(d)\times\cptwom$ and $\Sigma^k(d')\times\cptwom$
are diffeomorphic if and only if
$S^k\times\cptwom\,\#\,\Sigma^{4m+k}(d)$ and
$S^k\times\cptwom\,\#\,\Sigma^{4m+k}(d')$ are.
If we take connected sums with $\Sigma^{4m+k}(-d^\prime)$
we see that the preceding statements hold if and only if
$S^k\times\cptwom\,\#\,\Sigma^{4m+k}(d-d')$ and
$S^k\times\cptwom$ are diffeomorphic.  Therefore the proof of
Theorem~\ref{intro-thm: CPq cancells} reduces to showing that
{\sl the manifolds $S^k\times\cptwom\,\#\,\Sigma^{4m+k}(d-d')$ and
$S^k\times\cptwom$ are {\sf not} diffeomorphic if $d-d'$ is odd
(equivalently, $d-d'$ must be even if the manifolds {\sf are}
diffeomorphic).}

Suppose now that $S^k\times\cptwom\,\#\,\Sigma^{4m+k}(d-d')$ and
$S^k\times\cptwom$ are diffeomorphic.  Let $H:\Sigma^{4m+k}(d-d')\to
S^{4m+k}$ be the standard orientation-preserving homeomorphism
which is the ``identity'' on some coordinate disk, and
let $g$ be a connected sum of the identity on $S^k\times\cptwom$
with $H$.  Then the composite $h=g\tinycirc\varphi^{-1}$
defines a homotopy self-equivalence of $S^k\times\cptwom$ whose
image in the structure set ${\bf S}^s(S^k\times\cptwom)$ is
equal to $\Delta(d-d')$, where $\Delta$ is the Wall group action
map in the surgery sequence for $S^k\times\cptwom$.

By the discussion in the preceding paragraph, the proof of
Theorem~\ref{intro-thm: CPq cancells} reduces to showing that
{\sl if $k=7$, then there is no homotopy self-equivalence $h$ of
$S^k\times\cptwom$
whose class in the structure set is $\Delta(c)$, where $c$ is odd. }
Note that the normal invariant of such a self-equivalence must be
trivial by the exactness of the surgery sequence.

Our approach to studying this problem is to factor $h$ as a product
of self-equivalences which can be analyzed individually.  In
Chapter~\ref{sec: factorization} we describe such a factorization,
showing that $h$ is homotopic to a composite $f\tinycirc f'\tinycirc \theta$,
where $\theta$ is a diffeomorphism --- which we can ignore --- and two
homotopy self-equivalences $f$, $f^\prime$ coming from adjoints of
representatives for classes in $\pi_k(E_1(\cpq))$ and $[\cpq,E_1(S^k)]$,
respectively, where $E_1(Y)$ denotes the identity component in the
topological monoid of self-maps of a compact Hausdorff space $Y$ (with
the compact open topology).  To simplify the
discussion below, we shall simply say that $f$ and $f'$ {\sl come from}
$\pi_k(E_1(\cpq))$ {\sl and} $[\cpq,E_1(S^k)]$ respectively.

In general, if we are given a composite homotopy self-equivalence $f\tinycirc f'$
whose class in the structure set lies in the image of $\Delta$, then its
normal invariant vanishes by exactness, but there is no {\sl a priori\/}
reason why the normal invariants of $f$ and $f'$ must vanish.
By the composition formula for normal invariants,all we can say is that
the sum ${\mathfrak q}(f)+f^{\ast -1}{\mathfrak q}(f^\prime)$
vanishes.  However, in Section~\ref{sec: dichotomy proof and skeleton filtrations}
we shall prove that if $k=7$ then
their normal invariants ${\mathfrak q}(f)$ and ${\mathfrak q}(f^\prime)$ must
vanish for the factorization $h\simeq f\tinycirc f'\tinycirc\theta$ described
above.  Roughly speaking, the idea is to show that
${\mathfrak q}(f)$ and ${\mathfrak q}(f^\prime)$ lie in complementary
subgroups of $[S^k\times\cptwom, F/O]$.

The next step requires specialization to $k=7$, and it involves
proving that {\sl if $\alpha$ is a homotopy self-equivalence of
$S^k\times\cptwom$ coming from either $\pi_7(E_1(\cptwom))$ or $[\cptwom,E_1(S^7)]$,
then $\alpha$ is homotopic to a diffeomorphism if and only if its normal
invariant vanishes. }

If $\alpha$ comes from $[\cpq,E_1(S^7)]$, this follows quickly from the $\pi-\pi$
Theorem, and if $\alpha$ comes from $\pi_7(E_1(\cpq))$ with $q\geq 3$, then
by~\cite[Proposition 4.2]{Sch87} the vanishing
of ${\mathfrak q}(f)\in [S^7\times\cpq, F/O]$
implies that $\alpha$ is homotopic to the identity (and is {\sl a fortiori\/}
homotopic to a diffeomorphism).  For the sake of
completeness we shall give a self-contained proof of the latter in
Section~\ref{sec: order 2 nontrivial normal inv} (not all the proofs for results
used in~\cite[Section 4]{Sch87} have been published).  If $q=2$ and
$\alpha$ comes from $\pi_7(E_1(\cpq))$, then additional work is needed.
The main difference between this and the other cases is that the groups
$\pi_7(E_1(\cpq))\cong\Z_2$ are isomorphic
when $q\geq 3$ by stabilization theorems from~\cite{Sch73}, but the results
from the latter only show that there is a surjective stabilization map
$\pi_7(E_1(\cptwo)) \to \pi_7(E_1(\cpthree))$ whose kernel has order at most
2.  We are able to show that if a nontrivial unstable class exists, then
a representative homotopy self-equivalence $\alpha$ coming from
$\pi_7(E_1(\cptwo))$ must be homotopic to a diffeomorphism, and this is
the key step in proving that a homotopy equivalence $\alpha$ coming from
$\pi_7(E_1(\cptwo))$ is homotopic to a diffeomorphism if and only if its
normal invariant is trivial.

We can now bring everything together for the proof of
Theorem~\ref{intro-thm: CPq cancells} as follows:
Given a diffeomorphism $\varphi$ as above, define a homotopy self-equivalence
$h=g\tinycirc\varphi^{-1}$, where $g$ is described at the beginning of this
section.  Since $g$ is normally cobordant to the identity and $\varphi$ is a
diffeomorphism, it follows that $h$ is also normally cobordant to the identity.
We can now factor $h$ as a composite $f\tinycirc f'\tinycirc \theta$,
where $\theta$ is a diffeomorphism and $f$ and $f^\prime$ come from adjoints of
representatives for classes in $\pi_7(E_1(\cptwom))$ and $[\cptwom,E_1(S^7)]$
respectively.  As noted above, it follows that $f$ and $f^\prime$ are normally
cobordant to the identity and in fact are homotopic to diffeomorphisms.  We can
now use Taylor's result (Theorem~\ref{thm: taylor}) to conclude that $d-d'$
must be even.

\section{The Dichotomy Principle and its applications}
\label{sec: dihotomy}

For a general closed smooth manifold $M$, an understanding of normal
invariants for simple homotopy self-equivalences of $M$ is usually
indispensable to a diffeomorphism classification of smooth manifolds
which are simply homotopy equivalent to $M$.  We shall prove
the following strong but simply stated result for the manifolds
$M=S^7\times\cptwom$ (where $m\geq 1$).

\begin{thm}
{\bf (Dichotomy Principle)}\
\label{thm: main dichotomy}
If $f$ is a homotopy self-equivalence of
$S^7\times\cptwom$, then $f$ is homotopic to
a diffeomorphism if and only if $f$ has trivial normal invariant.
\end{thm}

In other words, there is a dichotomy:
Either $f$ is homotopic to a
diffeomorphism or else $f$ is not even normally
cobordant to the identity.
We shall prove Theorem~\ref{thm: main dichotomy}
later in Section~\ref{sec: dichotomy proof and skeleton filtrations},
and here we shall focus on its applications.

\begin{cor}
\label{cor: diffeo criterionfor sigma7(d)}
If $M=S^7\times\cptwom$, then
the number of oriented diffeomorphism types of
manifolds $\Sigma^7(d)\times \cptwom$ is equal to
the index of $I_h(M)\cap bP_{4m+8}$ in
$bP_{4m+8}$. Explicitly, the manifolds
$\Sigma^7(d)\times \cptwom$ and $\Sigma^7(d^\prime)\times \cptwom$
are orientation-preserving diffeomorphic
if and only if $\Sigma^{4m+7}(d-d^\prime)$ lies in
$I_h(M)$.
\end{cor}

\begin{proof}
First suppose that $\Sigma^7(d)\times\cptwom$ and
$\Sigma^7(d^\prime)\times\cptwom$ are orientation-preserving
diffeomorphic. By Fact~\ref{fact: browder}
this gives an orientation-preserving diffeomorphism
of $\Sigma^{4m+7}(d^\prime)\,\#\, M$ onto
$\Sigma^{4m+7}(d)\,\#\, M$, and
taking connected sum with $\Sigma^{4m+7}(-d^\prime)$,
we end up with an orientation-preserving
diffeomorphism $\phi$ of $M$ onto $\Sigma^{4m+7}(d-d^\prime)\,\#\, M$.

On the other hand, if
$H\co \Sigma^{4m+7}(d-d^\prime)\to S^{4m+7}$
is an orientation preserving homeomorphism which is the
``identity'' on some coordinate disk neighborhoods, then
the map $g:=H\,\#\, {\bf id}(M)\co
\Sigma^{4m+7}(d-d^\prime)\,\#\, M\to M$, which represents
$\Delta (d-d^\prime)$ in the structure set, has trivial normal invariant
by exactness of the surgery sequence.

The rest of the proof is similar to the argument sketched in the
previous section:  By the composition formula
for normal invariants ${\mathfrak q}(g\tinycirc\phi)$ is
trivial. Thus Theorem~\ref{thm: main dichotomy}
implies that $g\tinycirc\phi$
is homotopic to a diffeomorphism,
and it follows that $g$ is also homotopic to a diffeomorphism; {\it i.e.\/},
$\Sigma^{4m+7}(d-d^\prime)$ lies in the
homotopy inertia group $I_h(M)$, as claimed.

Conversely, if $\Sigma^{4m+7}(d-d^\prime)$ lies in the
homotopy inertia group $I_h(M)$, then
$\Sigma^{4m+7}(d-d^\prime)\,\#\, M$ is diffeomorphic to $M$, so
taking
connected sum with $\Sigma^{4m+7}(d^\prime)$, and an application of
Fact~\ref{fact: browder}
gives an orientation-preserving diffeomorphism
of $\Sigma^7(d)\times\cptwom$ and $\Sigma^7(d^\prime)\times\cptwom$.

Finally, the first assertion of the corollary follows because
the preimage of $I_h(M)$ under the map
$d\to\Sigma^{4m+7}(d)$ is a subgroup of $\mathbb Z$ whose index
is equal to
the index of $I_h(M)\cap bP_{4m+8}$ in
$bP_{4m+8}$, and $d-d^\prime$ is in this
subgroup if and only if $\Sigma^7(d)\times\cptwom$ and
$\Sigma^7(d^\prime)\times\cptwom$ are orientation-preservingly diffeomorphic.
\end{proof}

\begin{rmk}
\label{rmk: non-orient diffeo}
By composing with the product of ${\bf id}(\cptwom)$
and an orientation-reversing diffeomorphism
$\Sigma^7(d^\prime)\to \Sigma^7(-d^\prime)$, we immediately conclude
that $\Sigma^7(d)\times \cptwom$,
$\Sigma^7(d^\prime)\times \cptwom$ are orientation-reversingly
diffeomorphic if and only if $\Sigma^{4m+7}(d+d^\prime)$
lies in in $I_h(S^7\times \cptwom)$.
\end{rmk}

\begin{rmk}
As mentioned in Section~\ref{sec: surgery basics} the
standard homeomorphism from
$\Sigma^7(1)\,\#\,(S^3\times\cptwo)$ to $S^3\times\cptwo$
is not homotopic to a diffeomorphism, but the domain and codomain
are diffeomorphic. The proof of
Corollary~\ref{cor: diffeo criterionfor sigma7(d)}
then shows that $S^3\times\cptwo$ has a homotopy self-equivalence
with trivial normal invariant which is not homotopic
to a diffeomorphism.
\end{rmk}

\begin{proof}[Proof of Theorem~\ref{intro-thm: CPq cancells}]
If $\Sigma^7(d)\times\cptwom$ and
$\Sigma^7(d^\prime)\times\cptwom$ are diffeomorphic, then
by Corollary~\ref{cor: diffeo criterionfor sigma7(d)}
and Remark~\ref{rmk: non-orient diffeo} at least one of
the homotopy spheres
$\Sigma^{4m+7}(d-d^\prime)$ or $\Sigma^{4m+7}(d+d^\prime)$
lies in the homotopy inertia group
$I_h(S^7\times \cpq)$, which
contradicts Taylor's result (Theorem~\ref{thm: taylor})
because $d-d^\prime$ and $d+d^\prime=d-d^\prime+2d^\prime$ are odd.
\end{proof}

\section{Factorization of self-equivalences of $S^7\times\cpq$}
\label{sec: factorization}

This section describes some fairly canonical
factorizations for homotopy self-equivalences
of $S^k\times\cpq$, where $k$ is odd and $q\geq 1$.
The factors are given by a diffeomorphism and two
homotopy self-equivalences, each arising from a map
of one factor into the space of homotopy self-equivalences
of the other factor; similar results for a product of
two spheres $S^n\times S^k$ appear in \cite{Lev69}.
Our factorization plays an
important role in the proof of the Dichotomy Property.

As in Section \ref{sec: sketch},
given a compact Hausdorff space $T$ we let
${\mathcal E}(T)$ denote the group of all homotopy
classes of homotopy self-equivalences of $T$, and
$E_1(T)$ will denote the path-component of the identity
in the topological monoid of all self-maps of $T$ with
the compact-open topology.  If $T$ is homeomorphic to
a finite connected cell complex, then $E_1(T)$
has the homotopy type of a
$CW$-complex~\cite[Theorem 3]{Mil59}, and since $E_1(T)$ is arcwise
connected, standard results on $H$-spaces (see
\cite{Ja60}) imply that there is an inverse-up-to-homotopy map
$\rho:E_1(T)\to E_1(T)$
such that the self-maps of
$E_1(T)$ given by $f\to\rho(f)\tinycirc f$ and
$f\to f\tinycirc \rho(f)$ are homotopic to the
constant map $f\to {\bf id}(T)$.

For the rest of this section $X$ and $Y$ will
denote path-connected finite cell
complexes with base points $x_0$ and $y_0$ respectively.
The latter define slice inclusions $j(X),j(Y):X,Y\to X\times Y$,
and the coordinate projections onto $X$, $Y$ are denoted by
$p(X)$ and $p(Y)$ respectively.
Let ${\mathcal E}^\prime(X\times Y)$ be the set of
all classes $[f]\in {\mathcal E}(X\times Y)$
of homotopy self-equivalences such that
$p(X)\tinycirc f\tinycirc j(X)\simeq {\bf id}(X)$ and
$p(Y)\tinycirc f\tinycirc j(Y)\simeq {\bf id}(Y)$.

\begin{prop}
\label{prop: product cpq times sk}
Let $f:S^k\times\cpq\to S^k\times\cpq$ be
a homotopy self-equivalence, where $q\geq 1$ and $k$ is odd.
Then the following hold.

$(i)$
There is a diffeomorphism
$h:S^k\times\cpq\to S^k\times\cpq$ such that $f$ and $h$
induce the same automorphism of $H^*(S^k\times \cpq;\mathbb Z)$.

$(ii)$ If $f$ induces the
identity on $H^*(S^k\times \cpq;\mathbb Z)$, then
$[f]\in {\mathcal E}^\prime (S^k\times \cpq)$.

$(iii)$ ${\mathcal E}^\prime (S^k\times \cpq)$
equals the kernel of the
action of ${\mathcal E}(S^k\times \cpq)$ on cohomology.
\end{prop}

\begin{proof} $(i)$
The ring $H^*(S^k\times \cpq;\mathbb Z)$ is generated
by the classes of dimensions $2$ and $k$, which
also generate cohomology $2$nd and $k^{\rm th}$
cohomology groups, so the induced cohomology
automorphism $f^*$ of $H^*(S^k\times \cpq;\mathbb Z)$
is completely determined by its behavior on the
generators in dimensions $2$ and $k$, and it must be
multiplication by $\pm\,1$ in each case. If
$\chi$ is the conjugation involution on $\cpq$,
then ${\bf id}(S^k)\times\chi$ is multiplication
by $+1$ on the $k$-dimensional generator and
multiplication by $-1$ on the 2-dimensional generator,
while if $\varphi$ is reflection about a standard
$(k-1)$-sphere in $S^k$ then
$\varphi\times {\bf id}(\cpq)$ is multiplication
by $-1$ on the $k$-dimensional generator and
multiplication by $+1$ on the 2-dimensional generator.
Finally, the composition of these maps is multiplication by
$-1$ on both generators.  Thus every automorphism of
$H^*(S^k\times \cpq;\mathbb Z)$ is in fact induced by
a diffeomorphism.

{\it Proof of $(ii)$. }
The composite self-map of $S^k$ induces the identity
in cohomology and hence is homotopic to the identity;
similarly, the composite self-map of $\cpq$ also
induces the identity in cohomology, and a simple
obstruction-theoretic argument shows that this
composite must also be homotopic to the identity: indeed,
the restrictions to $\cpone$ are homotopic for degree
reasons, and the obstructions to
extending this to a homotopy of the original maps
lie in the groups
$H^{2j}\left(\cpq,\cpone;\pi_{2j}(\cpq)\,\right)$,
which are all trivial.

{\it Proof of $(iii)$. }
This follows immediately from $(ii)$ and the proof of $(i)$.
\end{proof}

Under the assumptions of Proposition~\ref{prop: product cpq times sk}
the set ${\mathcal E}^\prime (S^k\times \cpq)$ is a subgroup of
${\mathcal E}(S^k\times \cpq)$. 
In general, ${\mathcal E}^\prime(X\times Y)$ need not
be a subgroup of ${ \mathcal E}(X\times Y)$, but
regardless of whether or not this is true there are two important
subsets of ${ \mathcal E}'(X\times Y)$ which are subgroups;
each arises from a map
of one factor into the space of homotopy self-equivalences
of the other factor.
One of these subgroups is the image of a homomorphism
$\a_X\co [X,E_1(Y)]\to {\mathcal E}'(X\times Y)$
defined as follows:  Given a class
in $[X,E_1(Y)]$, choose a base point preserving representative
$g:X\to E_1(Y)$, where the base point of $E_1(Y)$ is the identity map;
then $g$ is adjoint to a continuous map
$g_{\#}:X\times Y\to Y$ whose restriction to $\{x_0\}\times Y$
is the identity.  Furthermore, if $g'$ is homotopic to $g$ then
$g_{\#}'$ is homotopic to $g_{\#}$
(this uses the fact that the adjoint isomorphism of function spaces
is a homeomorphism
$\mathfrak F\left( A,\,\mathfrak F (B, C)\,\right) \cong
\mathfrak F\left( A\times B, C\right)$
where ${\mathfrak F}$ denotes the continuous function space with
the compact open topology and $A,B,C$ are compact Hausdorff
spaces).

The existence of an inverse-up-to-homotopy self-map of $E_1(Y)$
implies the map $G$ that satisfies $p(Y)\tinycirc G = g_{\#}$
and $p(X)\tinycirc G = p(X)$ is a homotopy self-equivalence of $X\times Y$;
let $\alpha_X([g])$ denote the homotopy class of $G$.
Note that $\alpha_X([g])$ lies in
${ \mathcal E}'(X\times Y)$ because the
assumption that $g$ is base point preserving implies that
\[
p(Y)\tinycirc G\tinycirc j(Y) = {g_{\#}}\ \tinycirc j(Y)={\bf id}(Y)\qquad
p(X)\tinycirc G\tinycirc j(X) = p(X)\tinycirc
j(X)={\bf id}(X).
\]
Basic properties of adjoints imply that
$\alpha_X$ is a well-defined homomorphism into ${\mathcal E}(X\times Y)$
whose image lies in ${\mathcal E}^\prime(X\times Y)$.
Interchanging the roles of $X$ and $Y$ yields
a second homomorphism $\alpha_Y:[Y,E_1(X)]\to {\mathcal E}(X\times Y)$
with image in ${\mathcal E}^\prime(X\times Y)$.

Special cases of the following proposition are
in the literature ({\it e.g.\/}, in~\cite[2.5]{Lev69}).

\begin{prop}
\label{prop: factorization}
If ${\mathcal E}^\prime(X\times Y)$ is a subgroup of
${ \mathcal E}(X\times Y)$, then
every element in
${ \mathcal E}'(X\times Y)$ can be decomposed
as the product $\alpha_Y(v)\alpha_X(u)$ for
some $u\in [X,E_1(Y)]$ and $v\in [Y,E_1(X)]$.
\end{prop}
\begin{proof}
Suppose that $f$ represents an element of
${ \mathcal E}^\prime (X\times Y)$.
First, note that after changing $f$ within its homotopy
class we may assume $p(Y)\tinycirc f\vert_{\{x_0\}\times Y}$
equals $p(Y)\vert_{\{x_0\}\times Y}$.
To see this, use the homotopy lifting property for
$p_Y\co X\times Y\to Y$ to lift the homotopy between
$p(Y)\tinycirc f\vert_{\{x_0\}\times Y}$
and $p(Y)\vert_{\{x_0\}\times Y}$
to a homotopy joining $f\vert_{\{x_0\}\times Y}$
to the inclusion of $\{x_0\}\times Y$,
and then apply the homotopy extension property for
${\{x_0\}\times Y}$ to extend this homotopy
to a homotopy between $f$ and a map that
sends $(x_0, y)$ to $y$.

Let $U\co X\to\mathfrak F(Y, Y)$ be the adjoint of $p(Y)\tinycirc f$.
Since $U(x_0)=\mathbf{id}(Y)$ and $X$ is path-connected,
the image of $U$ lies in $E_1(Y)$, which lets us think of
$U$ as a map $X\to E_1(Y)$. Hence $U$ defines
a homotopy self-equivalence $g$ of $X\times Y$ given by
$g(x,y)=(x, U(x)(y))$; note that $p(Y)\tinycirc f=p(Y)\tinycirc g$.
A homotopy inverse to $g$ is given by
$g^\prime(x,y):=(x, \rho(U(x))(y))$, where $\rho$ is the
inverse-up-to-homotopy mentioned above.
This yields a homotopy joining
$p(Y)\tinycirc f\tinycirc g^\prime=
p(Y)\tinycirc g\tinycirc g^\prime$
with $p(Y)$, which lifts to a homotopy of
$f\tinycirc g^\prime$
and a map $(x,y)\to (V(y)(x), y)$
for some $V\co Y\to \mathfrak F(X,X)$.
In fact, the image of $V$ lies in $E_1(X)$ because by assumption
${\mathcal E}^\prime(X\times Y)$ is a subgroup,
and it clearly contains $f$ and $g^\prime$.
Since $f$ is homotopic to $f\tinycirc g^\prime\tinycirc g$,
the homotopy class of $f$ is equal to $\alpha_Y(v)\alpha_X(u)$
where $u$, $v$ are the homotopy classes of $U$, $V$, respectively.
\end{proof}

\begin{rmk}
Since $X\times Y$ and $Y\times X$ are canonically homeomorphic,
Proposition~\ref{prop: factorization} implies that if
${\mathcal E}^\prime(X\times Y)$ is a subgroup, then
the classes in ${\mathcal E}'(X\times Y)$ can be
also decomposed as $\alpha_X(u')\alpha_Y(v')$ for some
$u'$ and $v'$.
\end{rmk}

\section{A spectral sequence converging to $\pi_*(E_1(\cpq))$}
\label{sec: spectral seq}

By Propositions~\ref{prop: product cpq times sk}--\ref{prop: factorization}, 
if a homotopy self-equivalence of $S^7\times \cpq$ induces the identity
map on cohomology, then its homotopy class can be factored
as the composite of an element in the image of $\pi_7(E_1(\cpq))$ and
an element in the image of $[\cpq, S^7]$. This section
recalls and proves some results on $\pi_7(E_1(\cpq))$ using the
methods and results of \cite{Sch73}.

Let $F_{S^1}(\mathbb C^{q+1})$ be the space all
self-maps of $S^{2q+1}$, viewed as the unit sphere in $\mathbb C^{q+1}$,
which are $S^1$-equivariant with respect to the free linear action by
complex scalar multiplication; as usual, take the restricted compact
open topology on this equivariant function space and choose the
identity to be the base point.  The connectivity statement in
\cite[Theorem 5.7]{BecSch}
implies that $F_{S^1}(\mathbb C^{q+1})$ is connected for all $q\geq 1$.

Since the standard unitary group
action on $S^{2q+1}$ commutes with the free $S^1$-action by complex
scalar multiplication, it follows that
$U_{q+1}$ is a (compact) subgroup
of the topological monoid $F_{S^1}(\mathbb C^{q+1})$.  Furthermore,
since the quotient of the free linear $S^1$-action on $S^{2q+1}$ is the complex
projective space $\cpq$, there is an obvious
continuous homomorphism $\rho$ from $F_{S^1}(\mathbb C^{q+1})$ to
$E_1(\cpq)$ given by passage to quotients.
Results of James~\cite[Theorems 2.1--2.2]{Ja63} imply that
$\rho$ is a Serre fibration with the diagonally embedded
$S^1\subset U_{q+1}$ as a homotopy
fiber, and hence $\rho$ induces $\pi_k$-isomorphisms for all $k\ge 3$.
The same is also true for $k=2$ by
the proof of~\cite[Theorem 11.1]{BecSch}.

One important advantage of $F_{S^1}(\mathbb C^{q+1})$
over $E_1(\cpq)$ is the existence of the {\sl stabilization
homomorphism\/}
\[
s_{q+1}\co F_{S^1}(\mathbb C^{q+1})\to F_{S^1}(\mathbb C^{q+2})
\]
induced by taking the equivariant join of an equivariant self-map
of $S^{2q+1}$ with the identity on $S^1$~\cite[page 2]{BecSch}.
We denote the limit of these stabilizations by $F_{S^1}$ and the
limit of the iterated composites by $\sigma_{q+1}\co F_{S^1}(\mathbb C^{q+1})\to F_{S^1}$.

One of the most basic motivations for studying stabilization maps is
the {\sl Freudenthal Suspension Theorem\/}, which implies that
if $X$ is a CW complex then the
unreduced suspension homomorphisms
$$\pi_{k+r}(\Sigma^i X)~~\longrightarrow~~\pi_{k+i+1}(\Sigma^{i+1} X)$$
become isomorphisms for $i\geq k+2$ and epimorphisms for $i=k+1$
\cite[Theorem VII.7.13]{Wh}. For such values
of $i$ the group is the $k$-dimensional {\it stable homotopy group of $X$\/},
which we denote by $\pi_k^{\bf S}(X)$.

It is well known that the stabilization homomorphisms for the classical
group families $O_{q+1}\to O$ and $U_{q+1}\to U$ are respectively
$q$- and $(2q+2)$-connected.
The methods and results of \cite{Sch73} yield analogous information
about the stable range for the spaces $F_{S^1}(\mathbb C^{q+1})$
and $F_{S^1}$
(see the discussion on the top half of~\cite[p. 66]{Sch73}).
We shall use the following refinements of these results:

\begin{prop}\label{prop: stable range}
For all $q\geq 1$ the map $\sigma_{q+1}$ is $(2q+1)$-connected.
Furthermore, the kernel of the homomorphism
$$\sigma_{q+1\,*}\co \pi_{2q+1}\left(F_{S^1}(\mathbb C^{q+1})\,\right)~\longrightarrow~
\pi_{2q+1}\left(F_{S^1}\right)$$
is isomorphic to a quotient of the kernel for the stable suspension map from
$\pi_{4q+1}(S^{2q+1})$ to $\pi_{2q}^{\bf S}$.
\end{prop}

The Freudenthal Suspension Theorem implies that the
stabilization map from
$\pi_{4q+r}(S^{2q+r})$ to $\pi_{2q}^{\bf S}$ is an
isomorphism for $r\geq 2$, so the kernel of the stabilization map from
$\pi_{4q+1}(S^{2q+1})$ to $\pi_{2q}^{\bf S}$ is the image of the
Whitehead Product $[\iota_{2q+1},\iota_{2q+1}]$ by the exactness of
the EHP sequence \cite[Theorem XII(2.4)]{Wh}. This class has order at
most 2 by the anticommutativity of the Whitehead product, and thus the
kernel of the stabilization homomorphism from
$\pi_{2q+1}\left(F_{S^1}(\mathbb C^{q+1})\,\right)$ to
$\pi_{2q+1}\left(F_{S^1}\right)$ also has order at most 2.  In the cases
of interest for this paper, a slightly stronger conclusion holds.

\begin{cor}\label{cor: 7d stability}
The stabilization map from
$\pi_7\left(F_{S^1}(\mathbb C^{q+1})\,\right)$ to
$\pi_{7}\left(F_{S^1}\right)$ is an isomorphism when $q\geq 3$.
\end{cor}

\begin{proof}[Proof of Corollary~\ref{cor: 7d stability}]
When $q\geq 4$ this follows from Proposition~\ref{prop: stable range},
so it remains to verify the result when $q=3$, so that $2q+1=7$.
Since $S^7$ is an $H$-space ({\it e.g.\/}, by multiplication of unit
Cayley numbers), the Whitehead product $[\iota_7,\iota_7]\in
\pi_{13}(S^7)$ is trivial (compare \cite[Exercise 37, p. 392]{Hat}), and
therefore Proposition~\ref{prop: stable range}
and the preceding discussion imply that the stabilization map is both
one to one and onto.
\end{proof}

\begin{proof}[Proof of Proposition~\ref{prop: stable range}]
The argument on the top half of \cite[p. 66]{Sch73} proves that the
stabilization is $2q$-connected using the Zeeman version of the
Comparison Theorem
for spectral sequence mappings \cite[Theorem 1]{Zee}.  We shall first
show that a simple modification of that argument allows us to raise
the connectivity range by one dimension.
For the sake of concise notation we shall denote the spectral sequence for
$\pi_*\left(\,F_{S^1}(\mathbb C^{q+1})\,\right)$ by
$E^r_{s,t}\left(G\mathbb C^{q+1}\right)$ and the spectral sequence
for the stable analog $\pi_*\left(F_{S^1}\right)$ by
$E^r_{s,t}\left(G\mathbb C^{\infty}\right)$. The stabilization map
of spectral sequences will be denoted by $f^r_{s,t}$.

Zeeman's setting requires that the $E^2$-terms of a spectral
sequence are given by a short exact sequence involving tensor and torsion
products of $E^2_{0,*}$ and
$E^2_{0,*}$ similar to the Universal Coefficient Theorem, and
spectral sequence mappings are required to be presented similarly.
If the maps $g^2_{s,0}$ are isomorphisms for $s\leq P$ and the
maps $g^2_{0,t}$ are isomorphisms for $t\leq Q$, then there is an
isomorphism at the $E^\infty$ level through total degree $N=
{\rm min}(P-1,Q)$, and if $g^2_{0,N+1}$ is onto then the $E^\infty$
map in total degree $N+1$ is also onto.

The spectral sequences of interest in this proposition do not quite
fit into these hypotheses, but if we change the bigradings so that the
new term in bidegree $(s,t)$ is the old term in bidegree $(s+1,t)$,
then the conditions in Zeeman's paper are satisfied. Denote these
new spectral sequences by $E^r_{s,t}\left(\Omega G\mathbb C^{q+1}\right)$ and
$E^r_{s,t}\left(\Omega G\mathbb C^{\infty}\right)$, and let
$g^r_{s,t}=\Omega f^r_{s,t}$ denote the mapping of spectral sequences
corresponding to $f^r$ with this shift in dimensions.

For the modified spectral sequences in this proposition, we have
$P=2q+1$ and $Q=2q-1$, so that $N=2q$.  This implies that
the map $g^\infty$ is an isomorphism through total degree $2q-1$.
Furthermore, the Freudenthal Suspension Theorem implies that
the map $g^2_{0,2q}$ is onto, so Zeeman's result implies that
the map $g^\infty$ is onto in total degree $2q$.  Now the degrees
in the original and modified spectral sequences differ by 1,
so this means that the original spectral sequence mappings
$f^\infty$ are isomorphisms through total degree $2q$ and onto in
total degree $2q+1$.

The assertion about the kernel of $f^\infty$ in
total degree $2q+1$ follows because the methods of
\cite[proof of Theorem 2]{Zee} also show
that this kernel is contained in $E^\infty_{1,2q}$, so that it must be
isomorphic to a quotient of the kernel for the map
$$f^2_{1,2q}:
E^2_{1,2q}\left(G\mathbb C^{q+1}\right)\cong \pi_{4q+1}(S^{2q+1})~~
\longrightarrow~~E^2_{1,2q}\left(G\mathbb C^{\infty}\right)\cong
\pi_{2q}^{\bf S}$$
which according to \cite{Sch73} is given by stable suspension.
\end{proof}

The stabilization homomorphisms for the increasing sequences
$\left\{\,U_{q+1}\,\right\}$ and
$\left\{\,F_{S^1}(\mathbb C^{q+1})\,\right\}$
commute with the inclusions
$U_{q+1}\to F_{S^1}(\mathbb C^{q+1})$; and therefore define a continuous
map $U\to F_{S^1}$, where as usual $U$ denotes the direct limit of the
groups $U_{q+1}$.  The results of~\cite[Theorem 11.1]{BecSch}
(see also \cite{Seg70}) imply that
the induced maps of homotopy groups $\pi_*(U)\to\pi_*(F_{S^1})$ are split
injective and the cokernels are finite.

The main results of~\cite{BecSch} also yield another canonical
direct summand of $\pi_*(F_{S^1})$ as follows:  By~\cite[Theorem 6.6]{BecSch}
the homotopy groups $\pi_k(F_{S^1})$
are isomorphic to the stable homotopy groups
$\pi_k^{\bf S}(\Sigma\mathbb{CP}^\infty_+)$, where
$\Sigma\mathbb{CP}^\infty_+$ is the suspension
of the disjoint union of $\mathbb{CP}^\infty$ and a point.
Since $\Sigma\mathbb{CP}^\infty_+$ is equal to the wedge sum
$\Sigma\mathbb{CP}^\infty\vee S^1$, the group
$\pi_k^{\bf S}(\Sigma\mathbb{CP}^\infty_+)$
contains $\pi_k^{\bf S}(S^1)=\pi_{k-1}^{\bf S}$.  In
particular, this yields a subgroup of
$\pi_7^{\bf S}(\Sigma\mathbb{CP}^\infty_+)$
which is isomorphic to $\pi_6^{\bf S}\cong \mathbb Z_2$,
and the following result  shows that  $\pi_7(F_{S^1})$
is generated by this summand and the one described in the
previous paragraph:

\begin{lem}\label{lem: stable pi7}
The stable group $\pi_7(F_{S^1})\cong
\pi_7^{\bf S}(\Sigma\mathbb{CP}^\infty_+)$ is isomorphic to
$\mathbb Z\oplus \mathbb Z_2$, with the infinite cyclic
summand given by the image of $\pi_7(U)\cong\mathbb Z$ and
the $\mathbb Z_2$ summand given as in the preceding paragraph.
\end{lem}

\begin{proof}
It is only necessary to check that
$\pi_7^{\bf S}(\Sigma\mathbb{CP}^\infty)\cong
\pi_6^{\bf S}(\mathbb{CP}^\infty)$ is infinite cyclic,
for if this is the case then \cite[Theorem 11.1]{BecSch} implies
that generators for the image of this group in  $\pi_7(F_{S^1})\cong
\pi_7^{\bf S}(\Sigma\mathbb{CP}^\infty_+)$ and the
image of $\pi_7(U)$ agree up to sign and an element of finite order.
One way of proving that
$\pi_6^{\bf S}(\mathbb{CP}^\infty)\cong\mathbb Z$ is to use
\cite[Proposition 8.7]{Mukai} and the 7-connectedness of the inclusion
$\mathbb{CP}^3\subset \mathbb{CP}^\infty$.
\end{proof}

In contrast to Corollary~~\ref{cor: 7d stability}, the group
$\pi_7(F_{S^1}(\mathbb C^3))$ is unstable, and the
next goal of this section is to prove the following proposition:

\begin{prop}\label{prop: unstable homotopy group}
The map $s_{3\ast}\co\pi_7(F_{S^1}(\mathbb C^3))\to
\pi_7(F_{S^1}(\mathbb C^4))\cong\pi_7(F_{S^1})$
has kernel of order at most $2$, and
has image isomorphic to $\mathbb Z_2$.
If the kernel is nontrivial, then
$\pi_7(F_{S^1}(\mathbb C^3))\cong\mathbb Z_2\oplus\mathbb Z_2$.
\end{prop}

For the sake of concise notation we shall sometimes
denote $F_{S^1}(\mathbb C^{q+1})$ by $G\mathbb C^{q+1}$.

While Proposition~\ref{prop: unstable homotopy group}
does not compute $\pi_7(F_{S^1}(\mathbb C^3))$, its
conclusion together with Proposition~\ref{prop: eta composition}
will suffice for the purposes of this paper.

General tools for studying homotopy groups of
$F_{S^1}(\mathbb C^{q+1})$ are spectral sequences
developed in~\cite{Sch73} and~\cite{BecSch}.
To avoid additional digressions, we only use the spectral
sequences described in~\cite{Sch73}
and the relations among them.
These spectral sequences arise from the long exact homotopy
sequences associated to standard filtrations of function
spaces and classical Lie groups.
In all cases, the terms $E^1_{s,t}$ are the relative homotopy groups
$\pi_{s+t}({\rm Filt}^{(s)}, {\rm Filt}^{(s-1)})$, which
turn out to be canonically isomorphic to certain
homotopy groups of spheres.

In the case of $G\mathbb C^{q+1}$, the filtration is given by the
submonoids ${\rm Filt}^{(2p-1)}= {\rm Filt}^{(2p)}$ of
functions that restrict to inclusions on the standard subspheres
$S^{2q-2p+1}\subset S^{2q+1}$, where $0\leq p\leq q$;
by convention, ${\rm Filt}^{(2q+1)}$ is the entire space.
The associated spectral sequence $\{E^r_{s,t}(G\mathbb C^{q+1})\}$
is the homotopy spectral sequence that we have been studying in this
section.  By~\cite[Theorem 2.1]{Sch73} the spectral sequence converges to
$\pi_{s+t}(G\mathbb C^{q+1})$, and
$$E^2_{s,t}(G\mathbb C^{q+1})~~=~~H_{s-1}(\cpq,\pi_{t+2q+1}(S^{2q+1}))~.$$
To avoid confusion note that the $E^2$ term is described
in~\cite[Theorem 2.1]{Sch73} in terms of cohomology of $\cpq$,
and to obtain the above $E^2$ term we use
Poincar\'e duality for $\cpq$ and shift the bidegree from
$(s,t)$ to $(s+2q+1, t-2q-1)$ as in~\cite[Theorem 3.2]{Sch73}.

For the group $U_{q+1}$,
a similar filtration is given by the standardly embedded
unitary groups $U_{p}$, where $0\leq p\leq q$.
The associated spectral sequence
$\{E^r_{s,t}(U_{q+1})\}$ converges to
$\pi_{s+t}(U_{q+1})$, and
$$E^2_{s,t}(U_{q+1})~~=~~H_{s-1}(\cpq,\pi_{s+t}(S^{s}))$$
(see the discussion right before~\cite[Theorem 5.2]{Sch73}).
The inclusion $U_{q+1}\to G\mathbb C^{q+1}$
is compatible with these
filtrations, and~\cite[Theorem 5.2]{Sch73} yields
a canonical mapping of spectral sequences
which converges to
the map of homotopy groups induced by the inclusion
$U_{q+1}\to G\mathbb C^{q+1}$, and is given on the $E^2$ term
by the coefficient homomorphism induced by the
$(2q+1-s)$-fold suspension of homotopy groups of spheres.

Similarly, the stabilization map $s_{q+1}$
induces the map between the spectral sequences
for $G\mathbb C^{q+1}$ and $G\mathbb C^{q+2}$, which
by~\cite[Theorem 3.2]{Sch73} converges to
a map of homotopy groups induced by $s_{q+1}$, and which
corresponds on the $E^2$ term to
the homomorphism induced by the inclusion
$\cpq\to\cpqplusone$ and the double suspension on
coefficients.

Note that, as usual with homology spectral sequences,
the differentials on the $E^r$-term map
$E^r_{s,t}$ to $E^r_{s-r,\,t+r-1}$.

To prove Proposition~\ref{prop: unstable homotopy group}
we need some formulas for differentials
in above spectral sequences, where notation for
elements in the homotopy groups of spheres is the same
as in Toda's book~\cite{Tod}.

\begin{lem}
\label{lem: differentials in spec seq}
In the preceding spectral sequences, one has the following
differentials:\newline
\textup{(1)}
The differential $d^2_{5,0}(U_3):\pi_5(S^5)=\mathbb Z\to \pi_4(S^3)=
\mathbb Z_2$ sends the generator of the domain to the generator
of the codomain (which is the Hopf map $\eta_3:S^4\to S^3$).\newline
\textup{(2)}
The differential $d^2_{5,0}(G\mathbb C^{3}):\pi_5(S^5)=\mathbb Z\to
\pi_6(S^5)= \mathbb Z_2$ sends the generator of the domain to the generator
of the codomain (which is the Hopf map $\eta_5:S^6\to S^5$).\newline
\textup{(3)}
The differential $d^2_{5,2}(G\mathbb C^{3}):\pi_7(S^5)=\mathbb Z_2\to
\pi_8(S^5)= \mathbb Z_{24}$ is injective.\newline
\end{lem}

\begin{proof}  The validity of (1) follows because
this is the only choice of differential which is compatible
with the fact that $\pi_4(U_3)=0$; the latter holds because
the stabilization homomorphism $U_3\to U$ is 6-connected and
$\pi_4(U)=0$ by Bott periodicity.

Next, (2) follows because
the $E^2$-map from $\pi_4(S^3)=E^2_{3,1}(U_3)$ to $\pi_6(S^5)=
E^2_{3,1}(G\mathbb C^{3})$ is given by double
suspension~\cite[Theorem 5.2]{Sch73}, and
this map is bijective~\cite[Proposition 5.1]{Tod}.

To establish (3) it is enough to show that
$d^2_{5,2}(G\mathbb C^{3})$ is nontrivial on
$\pi_7(S^5)=E^2_{5,2}(G\mathbb C^{3})$, which
is generated by the square $\eta^2$ of the Hopf
map~\cite[Proposition 5.3]{Tod}. To this end it helps
to use composition operations of the spectral sequence
as described in~\cite[Proposition 1.4]{Sch73} (also
see the background discussion on pp. 53--54).
Denoting the identity element of $S^5$ by $1$, and
thinking of $\eta^2$ as $1\tinycirc\eta^2$, we write
$d^2_{5,2}(G\mathbb C^{3})(\eta^2)$ as
$d^2_{5,0}(G\mathbb C^{3})(1\tinycirc \eta^2)$
which stably is equal to $d^2_{5,0}(G\mathbb C^{3})(1)\tinycirc\eta^2$
because the operation of precomposing with $\eta^2$ stably
commutes with differentials. Now (2) implies
that $d^2_{5,2}(G\mathbb C^{3})(\eta^2)=
\eta\tinycirc\eta^2=\eta^3$ which has order $2$ in
$\pi_8(S^5)=E^2_{3,3}(G\mathbb C^{3})$~\cite[formula (5.5), page 42]{Tod};
thus $d^2_{5,2}(G\mathbb C^{3})$ is nontrivial.
\end{proof}

The proofs of Proposition~\ref{prop: unstable homotopy group} and
other results will also require the following input:

\begin{prop}\label{prop: bottom filtration}
In the spectral sequences for the groups
$\pi_*\left(F_{S^1}(\mathbb C^{q+1})\right)$
we have
$$E^\infty_{1,6}(G\mathbb C^{q+1})~~\cong~~
E^2_{1,6}(G\mathbb C^{q+1})~~\cong~~\mathbb Z_2$$
for all $q\geq 2$, and in this range the stabilization homomorphisms
$E^\infty_{1,6}(G\mathbb C^{q+1})\to E^\infty_{1,6}(G\mathbb C^{q+2})$
are isomorphisms.
\end{prop}

\begin{proof}
By \cite[Section 3]{Sch73}, on the $E^2$ level the stabilization homomorphisms
$$
E^2_{1,6}(G\mathbb C^{3})~~\to~~E^2_{1,6}(G\mathbb C^{4})~~\cdots~~\to~~
E^2_{1,6}(G\mathbb C)~~\cong~~\mathbb Z_2
$$
are equivalent to the double and stable suspension homomorphisms
$$\pi_{11}(S^5)~~\to~~\pi_{13}(S^7)~~\cdots~~\to~~
\pi_{6}^{\bf S}~~\cong~~\mathbb Z_2$$
and the latter are all isomorphisms by \cite[Chapter V]{Tod}.

As before, let ${\rm Filt}^1(G\mathbb C^{q+1})\subset F_{S^1}(\mathbb C^{q+1})$
denote the set of all equivariant self-maps whose restriction to $S^{2q-1}$,
viewed as the unit sphere in $\{{\bf 0}\}\times\mathbb C^q\subset \mathbb C^{q+1}$,
is the inclusion mapping.  The results of \cite[Section 1]{Sch73} imply that
${\rm Filt}^1(G\mathbb C^{q+1})$ is homotopy equivalent to the iterated loop
space $\Omega^{2q} S^{2q+1}$ and the induced maps in homotopy groups
$$\pi_{t+2q}(S^{2q+1})~~\cong~~\pi_t(\Omega^{2q} S^{2q+1})~~\cong~~
\pi_t({\rm Filt}^1(G\mathbb C^{q+1}))~~\to~~\pi_t(F_{S^1}(\mathbb C^{q+1}))$$
correspond to the edge homomorphisms
$$
E^2_{1,t-1}(G\mathbb C^{q+1})~~\to~~E^\infty_{1,t-1}(G\mathbb C^{q+1})~~\subset~~
\pi_t(F_{S^1}(\mathbb C^{q+1}))
$$
where the left map is given by the projection  $E^2\to E^\infty$ in the lowest
nontrivial filtration ({\it i.e.\/}, degree 1).  Therefore it will suffice to
prove that for all $q\geq 3$ the composite
$$
{\mathbb Z_2}\cong\pi_7({\rm Filt}^1(G\mathbb C^{q+1}))~~\to~~
\pi_7(F_{S^1}(\mathbb C^{q+1}))~~\to~~\pi_7(F_{S^1})~~\cong~~
{\mathbb Z}_2\oplus{\mathbb Z}
$$
is injective.

Following \cite[p. 8]{BecSch}, let $F_{S^1}({\mathbb C}|\mathbb C^{q+1})$
denote the space of $S^1$-equivariant self-maps from $S^{2q-1}$ (again
viewed as the unit sphere in $\{{\bf 0}\}\times\mathbb C^q\subset \mathbb C^{q+1}$)
to $S^{2q+1}$, and consider the restriction mapping $\rho$ from
$F_{S^1}(\mathbb C^{q+1})$ to $F_{S^1}({\mathbb C}|\mathbb C^{q+1})$. By the
methods of \cite{Sch73} this is a fibration whose fiber is
${\rm Filt}^1(G\mathbb C^{q+1})\cong \Omega^{2q} S^{2q+1}$.  If we take the
limits of stabilizations and apply \cite[Theorem 6.6 and (6.8)]{BecSch}, we obtain the
following commutative diagram whose component parts are described below.
Each column in the diagram represents a fibration.

$$
\begin{matrix}
{\rm Filt}^1(G\mathbb C^{q+1})\simeq \Omega^{2q} S^{2q+1} &
\overset{\sigma'}{\longrightarrow} &
{\rm Filt}^1(G\mathbb C)\simeq {\bf Q}(S^1) &
\overset{\lambda'}{\longrightarrow} &
{\bf Q}(S^1)\\
\downarrow i&&\downarrow i_\infty &&\downarrow j \\
F_{S^1}(\mathbb C^{q+1}) & \overset{\sigma}{\longrightarrow} &
F_{S^1} & \overset{\lambda}{\longrightarrow} &
{\bf Q}\left(\Sigma\mathbb{CP}^\infty_+\right)\\
\downarrow \rho&&\downarrow \rho_\infty &&\downarrow p \\
F_{S^1}({\mathbb C}|\mathbb C^{q+1}) & \overset{\sigma_0}{\longrightarrow} &
\displaystyle{\lim_{\;m\to\infty}\! F_{S^1}({\mathbb C}|\mathbb C^{m+1})} &
\overset{\sigma_0}{\longrightarrow} &
{\bf Q}\left(\Sigma\mathbb{CP}^\infty\right)
\end{matrix}
$$

The previously undefined maps and spaces in this diagram are given as follows:

\begin{enumerate}
\item If $X$ is a space then ${\bf Q}(X)$ is the free infinite loop space
$\Omega^\infty\Sigma^\infty X$.
\item The mapping $\rho_\infty$ is the stable analog of the restriction
mapping $\rho$.
\item The mapping $i$ is the inclusion of a relative
function space (specified behavior on a given subspace) into the entire
function space, and  $i_\infty$ is the stable analog of $i$.
\item The mappings  $\sigma$ and $\sigma_0$ are stabilizations, and
$\sigma'$ is stable suspension (the third assertion follows from the commutative
diagram and \cite[Theorem 3.2]{Sch73}).
\item  The mapping $p$ is the first factor coordinate projection on
${\bf Q}\left(\Sigma\mathbb{CP}^\infty_+\right)\simeq
{\bf Q}\left(\Sigma\mathbb{CP}^\infty\right)\times {\bf Q}(S^1)$,
and the mapping $j$ is the second factor slice injection into
that space.
\item  The homotopy equivalences $\lambda$ and $\lambda_0$ are defined as
in \cite[Section 5]{BecSch}, and the fiber map $\lambda'$ is defined using the
homotopy commutativity identity $p\tinycirc\lambda = \lambda_0\tinycirc
\rho_\infty$ which follows from \cite[(6.8)]{BecSch}.  A standard
topological Five Lemma argument (see \cite[unnumbered Proposition on p. 80]{HP}
for the semisimplicial version) implies that $\lambda'$ is also a homotopy
equivalence.
\end{enumerate}

The proof of the proposition now follows quickly.  By the diagram above,
the composite
$$\lambda_*\tinycirc\sigma_*\tinycirc i_*~=~j_\ast\tinycirc\lambda'_*\tinycirc\sigma'_*:
\pi_{2q+7}(S^{2q+1})~\to~\pi_6^{\bf S}~\to~
\pi_7^{\bf S}\left(\Sigma\mathbb{CP}^\infty_+\right)$$
is given by the stabilization $\sigma_*$ followed by the isomorphism
$\lambda'_*$ (recall that $\lambda'_*$ is a homotopy equivalence) followed
by the split injection $j_*$ (note that the slice inclusion $j$ is a
retract).  In the first paragraph of the proof we observed
that $\sigma^\prime_*$ is an isomorphism by results from \cite{Tod}, so
this proves that the displayed composite is
injective, which is what we needed to complete the proof of the proposition.
\end{proof}

We now have enough information to proceed with our analysis of
$\pi_7\left(F_{S^1}(\mathbb C^{q+1})\right)$.

\begin{proof}[Proof of Proposition~\ref{prop: unstable homotopy group}]
Lemma~\ref{lem: differentials in spec seq}(3) implies that
$E^\infty_{s,7-s}(G\mathbb C^{3})=0$ except possibly when $s\neq 1,3$,
and from Proposition~\ref{prop: bottom filtration}
we see that $E^\infty_{1,6}(G\mathbb C^{3})=\mathbb Z_2$, which maps injectively
into $\pi_7(F_{S^1})$.  Therefore, the only uncertainty
involves the group $E^\infty_{3,4}(G\mathbb C^{3})$.
Recall that
$E^2_{3,4}(G\mathbb C^{3})=
\pi_9(S^5)\cong\mathbb Z_2$~\cite[Proposition 5.8]{Tod}
and $E^2_{3,4}(G\mathbb C^4)=0$.
Hence the quotient of $\pi_7(F_{S^1}(\mathbb C^3))$
by the subgroup $E^\infty_{1,6}(G\mathbb C^{3})\cong\mathbb Z_2$ is
$E^\infty_{3,4}(G\mathbb C^{3})$, which is a group of order at most $2$.
Since $E^\infty_{3,4}(G\mathbb C^4)=0$, the stabilization
homomorphism maps $\pi_7(F_{S^1}(\mathbb C^3))$ onto the unique
order two subgroup of $\pi_7(F_{S^1})\cong\mathbb Z\oplus\mathbb Z_2$
that is the image of $E^\infty_{1,6}(G\mathbb C^{3})$.
If $E^\infty_{3,4}(G\mathbb C^{3})\cong\mathbb Z_2$,
then the group $\pi_7(F_{S^1}(\mathbb C^3))$ has order $4$,
and it cannot be isomorphic to $\mathbb Z_4$ because it
has an order $2$ element that does not lie in the kernel
of the homomorphism into $\pi_7(F_{S^1})$;
thus in this case
$\pi_7(F_{S^1}(\mathbb C^3))\cong\mathbb Z_2\oplus\mathbb Z_2$.
\end{proof}

We shall also need the following information about the kernel of
the stabilization map $s_{3*}$ in
Proposition~\ref{prop: unstable homotopy group}:

\begin{prop}\label{prop: eta composition}
The kernel of $s_{3\ast}\co\pi_7(F_{S^1}(\mathbb C^3))\to
\pi_7(F_{S^1}(\mathbb C^4))\cong\pi_7(F_{S^1})$
is contained in the image of the map
$$\eta_6^*:\pi_6(F_{S^1}(\mathbb C^3))\to\pi_7(F_{S^1}(\mathbb C^3))$$
defined by composition with the class of the suspended Hopf map
$\eta_6\in\pi_7(S^6)$.
\end{prop}

\begin{rmk}
\label{rmk: composition additivity}
If $\alpha\in\pi_p(S^q)$ is given and $X$ is an arcwise connected topological
space, then the map $\alpha^*:\pi_q(X)\to\pi_p(X)$ induced by composition is
not always additive (the simplest example is the Hopf class $\eta_2\in
\pi_3(S^2)$ acting on $\pi_2(S^2)\cong\mathbb {Z}$, where $\eta_2^*(k)=
k^2\eta_2)$, but $\alpha^*$ is additive if either $X$ is an $H$-space or
$\alpha$ desuspends to $\pi_{p-1}(S^{q-1})$.  The first condition holds
in the setting of the proposition, and in general the second condition
follows from the first and the adjoint isomorphism $\pi_k(X)\cong
\pi_{k-1}(\Omega X)$.
\end{rmk}

\begin{proof}[Proof of Proposition~\ref{prop: eta composition}]
Throughout this proof, {\sl we shall assume that the kernel of the stabilization
homomorphism $s_{3\ast}$ is nontrivial\/}.  Under this hypothesis, the proof of of
Proposition~\ref{prop: unstable homotopy group} implies that the kernel has order 2
and a nonzero representative corresponds to a nonzero class in
$E^\infty_{3,4}(G\mathbb C^{3})=E^2_{3,4}(G\mathbb C^{3})=
\pi_9(S^5)\cong\mathbb Z_2$.  In particular, we must have $d^2_{3,3}=0$ and
$d^2_{5,2}=0$; note that all higher differentials to or from the groups
$E^r_{3,\ast}$ must vanish (either the domain or codomain
is zero for dimensional reasons) and hence we have $E^3_{3,\ast}(G\mathbb C^{3})\cong
E^\infty_{3,\ast}(G\mathbb C^{3})$.

The composition operations
$$\eta_8^*:\pi_8(S^5)~\to~\pi_9(S^5)~~\cong~~{\mathbb Z}_2~,\qquad
\eta_9^*:\pi_9(S^5)~\to~\pi_{10}(S^5)~~\cong~~{\mathbb Z}_2$$
are surjective and bijective by \cite[Propositions 5.8 and 5.9]{Tod}.
Since the composition operations for the
spectral sequences $E^r_{s,t}(G{\mathbb C}^{3})$ satisfy
$\eta_9^*\tinycirc d^2_{3,3}~=~d^2_{3,4}\tinycirc\eta_8^*$
by \cite[Proposition 1.4]{Sch73}, it follows that $d^2_{3,3}(G\mathbb C^{3})$
is zero if $d^2_{3,4}(G\mathbb C^{3})$ is zero.  Our hypothesis implies that
the latter differential is zero, so it follows that the generator of
$E^2_{3,3}(G{\mathbb C}^{3})\cong\mathbb{Z}_{24}$ must be
a permanent cycle ({\it i.e\/}, it survives to $E^\infty$).  This class is
not a boundary, for the image of the differential $d^2_{5,2}(G\mathbb C^{3})$
which maps to $E^2_{3,3}(G{\mathbb C}^{3})$ has order 2 by part (3) of
Lemma~\ref{lem: differentials in spec seq}.

If $\varphi\in \pi_6
(F_{S^1}(\mathbb C^3))$ represents this permanent cycle in
$E^\infty_{3,3}(G{\mathbb C}^{3})$, then
\cite[Proposition 1.4]{Sch73} implies that $\eta_6^*\phi$ represents
the generator of
$$E^\infty_{3,4}(G{\mathbb C}^{3})~~\cong~~E^2_{3,4}(G{\mathbb C}^{3})~~
\cong~~\pi_9(S^5)~~\cong~~\mathbb{Z}_2~.$$
To prove that $\eta_6^*\varphi$ must be stably trivial, we shall use
the homotopy equivalence $\lambda:F_{S^1}\to \Omega^\infty S^\infty
(\Sigma\mathbb{CP}^\infty_+)$ of \cite{BecSch}, where $X_+$ denotes the space
$X\vee S^0$.  Specifically, it will suffice to check that the homomorphism
$$\eta^*:\pi^\mathbf{S}_6\left(\Sigma\mathbb{CP}^\infty_+\right)~\to~
\pi^\mathbf{S}_7\left(\Sigma\mathbb{CP}^\infty_+\right)$$
is zero, and since $X_+=X\vee S^0$ this reduces to the corresponding
statements for the summands $\Sigma\mathbb{CP}^\infty$ and $S^1$.  In the
second case this follows because $\pi^\mathbf{S}_5=0$ by
\cite[Proposition 5.9]{Tod}, while in the first this follows because
$\pi^\mathbf{S}_6\left(\Sigma\mathbb{CP}^\infty\right)$ is finite and
$\pi^\mathbf{S}_7\left(\Sigma\mathbb{CP}^\infty\right)$ is infinite cyclic
(see \cite{Mukai}, Proposition 8.7, and recall that the inclusion of
$\mathbb{CP}^3$ in $\mathbb{CP}^\infty$ is 7-connected).
\end{proof}

\section{Structure sets and self-equivalence spaces}
\label{sec: self and structure sets}

Sections 5 and 6 allow us to work with homotopy self-equivalences
of $S^7\times\cpq$ very effectively, and in Sections 8--10 we
shall use these results to prove the Dichotomy Principle
(Theorem~\ref{thm: main dichotomy}) which characterizes the homotopy
self-equivalences of $S^7\times\cpq$ that are homotopic to diffeomorphisms.
The results of this section are the basis for studying the image of a
homotopy self-equivalence $h$ in the Sullivan-Wall structure set
${\bf S}^{s}(S^m\times\cpq)$ when $h$ comes from a class in the homotopy
group
$\pi_m(E_1(\cpq ))\cong\pi_m(F_{S^1}\left(\,\mathbb C^{q+1})\,\right)$,
where $m\ge 2$.

If $X$ is a closed smooth manifold
and $m\geq 1$, then it is well known that the relative
structure sets
$${\bf S}^{s}(D^m\times X~\text{rel}~S^{m-1}\times X)~~=:~~
{\bf S}_m^{s}(X)$$
have group structures given by a ``stacking'' operation analogous to
the group structure for homotopy groups ({\it e.g.\/}, see
\cite{Bro-Pet, Pet70, Roth70}).
The definition of this operation
is obvious for $m=1$, and for $m\geq 2$
follows by the recursive application of the case $m=1$; in analogy with
homotopy groups, the operation is commutative if $m\geq 2$. Furthermore,
if $n=m+\dim X\geq 5$ then all maps in the surgery exact sequence
\[
\xymatrix{
L^s_{n+1}\left(\,\pi_1(X)\,\right)\to
{\bf S}^{s}_m(X)\to
\left[(D^m\times X)/(S^{m-1}\times X),F/O\right]\to
L^s_{n}\left(\,\pi_1(X)\,\right)
}
\]
are homomorphisms of groups.  The terms in the sequence involving $F/O$
simplify to $[\Sigma^m X,F/O]\oplus \pi_m(F/O)$ (where $\Sigma^m$ denotes
$m$-fold suspension) because there is a natural homeomorphism from the
quotient $(D^m\times X)/(S^{m-1}\times X)$ to $\Sigma^m X\vee S^m$ ({\it cf.\/}
\cite[Lemma 2.4 and the definition on p. 296]{At}).

The group operations on the sets ${\bf S}^{s}_m(X)$ are extremely useful
for studying the structure sets ${\bf S}^{s}(S^m\times X)$
because there is a canonical mapping
\begin{equation}
\label{form: map Gamma}
\Gamma\co {\bf S}_m^s(X)\to {\bf S}^s(S^m\times X)
\end{equation}
that is obtained by {\sl extension by diffeomorphism\/}:
Given a representative
$h:(W,\partial W)~\longrightarrow~(D^m\times X, \d D^m\times X)$
of a relative structure, where the boundary map
$\partial h$ is a diffeomorphism (resp. homeomorphism), we set
$V$ equal to $W~\cup_{\partial h}~D^m\times X$, where
the $D^m$-factor is given the opposite orientation,
and let $V\to S^m\times X$ be the well-defined map
which is given by $h$ on $W$ and the identity on $D^m\times X$.
This construction preserves homotopies through maps that are
diffeomorphisms on the boundary, and hence it
yields a well-defined map $\Gamma$. Furthermore, the construction
preserves the identity structure, and hence $\G$ preserves the base points
(recall that while ${\bf S}_m^s(X)$ is a group,
${\bf S}^s(S^m\times X)$ is merely a pointed set).

Suppose now that we have a homotopy self-equivalence $h$ of $S^m\times X$
whose homotopy class in ${\mathcal E}^\prime(S^m\times X)$
comes from $\pi_m(E_1(X))$.  We claim that the associated simple
homotopy structure in ${\bf S}^s(S^m\times X)$ lies in the image of
${\bf S}_m^s(X)$.  This is true because every element of
$\pi_m(E_1(X))$ can be represented by a map
$(D^m, \d D^m=S^{m-1})\to E_1(X)$ such that all points in a small
neighborhood of $\d D^m$ are mapped to ${\bf id}(X)$.
If $m\ge 1$, so that ${\bf S}^s_m(X)$ is equipped with the stacking
group structure, this construction yields a group homomorphism
\begin{equation}
\label{form: map Psi}
\Psi\co\pi_m(E_1(X))~~\longrightarrow~~{\bf S}_m^{s}(X)
\end{equation}
(which equals $\sigma_m\tinycirc k$ in the notation of
\cite[Sections~3.3 and 3.6]{ABK}) such that
the standard map $\pi_k(E_1(X))\to {\bf S}^{s}(S^m\times X)$ defined
via adjoint can be factored as $\Gamma\tinycirc\Psi$.

In many cases the classes in the image of $\Gamma\tinycirc\Psi$ represent
{\sl tangential\/} homotopy self-equivalences of $S^m\times X$; recall
that if $Y$ is a smooth manifold (possibly with boundary), then a homotopy
self-equivalence $h$ of $(Y,\partial Y)$ is said to be tangential if and
only if any (hence all) of the following hold:

\begin{enumerate}
\item If $\nu_Y$ is the normal bundle of $Y$ in some Euclidean space,
then the vector bundles $\nu_Y$ and $h^*\nu_Y$ are stably isomorphic.

\item If $\tau_Y$ is the tangent bundle of $Y$,
then the vector bundles $\tau_Y$ and $h^*\tau_Y$ are stably isomorphic.

\item  The normal invariant of $h$ goes to zero under the canonical
homomorphism $[Y,F/O]\to [Y,BO]$.

\item  The normal invariant of $h$ lies in the image of the canonical
homomorphism $[Y,F]\to [Y,F/O]$.
\end{enumerate}

(The first two statements are equivalent because the images of $\tau_Y$
and $\nu_Y$ are negatives of each other in ${\widetilde {KO}}(Y)\cong
[Y,BO]$, the third is equivalent to the first two because the normal
invariant's image in $[Y,BO]$ is the stable difference class
of $\nu_Y-(h^*)^{-1}\nu_Y$, and the fourth is equivalent to the
third by the exactness properties of the fibration $F\to F/O\to BO$.)

In particular, classes in the image of $\Gamma\tinycirc\Psi$ represent
tangential homotopy self-equivalences of $S^m\times X$ whenever
$X$ is stably parallelizable because $\nu=0$ then (this essentially
goes back to \cite[Section 7]{nov}).  Here is another class of examples
which will be useful for our purposes:

\begin{prop}\label{prop: tangentials}
If $q\geq 1$ and $m\geq 2$, then all classes in the image of
$$\Gamma\tinycirc\Psi\co\pi_m(E_1(\cpq))~~\longrightarrow~~{\bf S}^{s}(S^m\times\cpq)$$
are tangential.
\end{prop}

{\bf Remark.}
For completeness, we note that this result is also valid when $m=1$.
In fact, if we embed $PU_{q+1}$
into $E_1(\cpq)$ using the action by projective unitary transformations,
then Theorem 2.1 of \cite{Ja63}, the subsequent discussion on page 47 of that paper,
the results of \cite{BecSch}, and the commutative diagram
\[
\xymatrix{
S^1\ar[r]^\subset\ar[d]^{\simeq}&
U_{q+1}\ar[r]\ar[d]^{\subset}&
PU_{q+1}\ar[d]^{\subset}&\\
\mathrm{Maps\,}(\cpq,S^1)\ar[r]^\subset&
F_{S^1}(\mathbb C^{q+1})\ar[r]&
E_1(\cpq)
}
\]
combine to imply that $\pi_1(PU_{q+1})\to \pi_1(E_1(\cpq)\,)$ is
an isomorphism, so that every element of the latter is represented
by a map coming from a diffeomorphism of $S^1\times\cpq$ and the result
follows because diffeomorphisms are always tangential.  However,
this exceptional case will not be needed here.

\begin{proof}[Proof of Proposition~\ref{prop: tangentials}]
Since $m\geq 2$ the map of homotopy groups induced by the orbit space
projection from
$\pi_m\left(\,F_{S^1}(\mathbb C^{q+1})\,\right)$ to $\pi_m(E_1(\cpq))$
is an isomorphism, so we can represent a class in the latter group by
a continuous $S^1$-equivariant map $g:D^m\times S^{2q+1}\to S^{2q+1}$ whose restriction
to a neighborhood of the boundary is projection onto the second
coordinate.  Let $g'$ denote the associated equivariant homotopy self-equivalence of
$D^m\times S^{2q+1}$ whose projections onto the first and second
factors are $(p_D,g)$ where $p_D$ is the usual factor projection onto $D^k$.
By construction this equivariant self-equivalence is the identity near the
boundary, and of course the same is true when one passes to the orbit space
$D^m\times\cpq$.  
Now the extension by diffeomorphism construction,
described after (\ref{form: map Gamma}),
gives a homotopy self-equivalence of $S^m\times S^{2q+1}$.

By the second characterization of tangential self-equivalences in the list above,
it will suffice to construct a map of stable tangent bundle total spaces
$$G:T(D^m\times\cpq)~~\longrightarrow~~T(D^m\times\cpq)$$
which is a linear isomorphism on each fiber, covers the induced map $g/S^1$
of orbit spaces, and is the identity near the
inverse image of the boundary $S^{m-1}\times\cpq$ in $T(D^m\times\cpq)$.
We can do this explicitly as follows:
The stable complex tangent bundle of $\cpq$ is isomorphic to a direct sum of
$(q+1)$ copies of the dual to
canonical line bundle ({\it e.g.\/}, see \cite{Milnor-Stasheff}) with
total space $S^{2q-1}\times_{S^1}(\mathbb{C}^{q+1})^*$, and therefore
the balanced product
$$G~~=~~\left(g'\times_{S^1}\mathbf{id}((\mathbb{C}^{q+1})^*)\right)\times
\mathbf{id}(\mathbb{R}^k)$$
defines a map of stable tangent bundles covering $g'$, and such that the
restriction to a neighborhood of
$S^{k-1}\times \left(S^{2q-1}\times_{S^1}(\mathbb{C}^{q+1})^*\right)\times\mathbb{R}^k$
is the identity.
\end{proof}

{\bf Composition operations on structure sets.}
If the normal invariant of a homotopy self-equivalence is trivial, it is
often very difficult to determine whether or not its class in the
structure set is trivial, and successful computations require a broad
assortment of techniques.  We conclude this section with a homotopy-theoretic
method which works in the case needed in
Section~\ref{sec: stably trivial is trivial}
(see~\ref{cor: stably trivial B}).  In an earlier version of this paper, we
analyzed this case using methods related to~\cite{MTW}.

It is well known that the long exact surgery sequence
in~\cite[Chapter 10]{Wal-book} is realized as the exact sequence of homotopy groups
associated to a fibration of $\Delta$-{\sl sets} in the sense of \cite{RoSa}:
$$
\cS_{\bullet}^{{DIFF}, x} (M^n) \to \cF_{\bullet} (M^n, F/O) \to
L^x_{\bullet} \left(\pi_1 (M), w_1 (M^n)\right)
$$
Here $M^n$ is a closed smooth manifold of dimension $n\geq 5$ and the
superscript $x$ in our case is $s$ for simple homotopy structures or $h$ for
homotopy structures ({\it e.g.\/}, see Rourke~\cite{Ro} or Quinn~\cite{Q1}).
A $k$-simplex in $\cS^{{DIFF}, x}_{\bullet} (M^n)$ is represented
by a suitable homotopy equivalence of manifold
$n$-ads~\cite[Chapter 0]{Wal-book} into the standard $n$-ad given by
$\Delta_k \times M$, the set $\cF_{\bullet} (M^n, F/O)$ is the simplicial
function set as defined in May~\cite[Definition I.6.4, p. 17]{My} or
Goerss-Jardine~\cite[p. 20]{GJ}, and a $k$-simplex
of $L_{\bullet}^x (-)$ is represented
by a suitable surgery problem of manifold $n$-ads.  If $n\leq 4$ then the
same constructions still yield the $\Delta$-sets $\cS^{{DIFF}, x}_{\bullet}
(M^n)$ with analogous mappings from these sets to the function sets
$\cF_{\bullet} (M^n, F/O)$.

Let  $M^n$ be as above, with no restriction on its dimension.  If
$\cH_\bullet(M^n)$ is the subobject of the simplicial function monoid
$\cF_{\bullet} (M^n, M^n)$ defined by restricting to self-maps
which are homotopic to the identity, then there is a canonical
map $\Psi_\bullet$
from $\cH_\bullet(M^n)$ to $\cS^{{DIFF}, s}_{\bullet}
(M^n)$ given by taking a homotopy equivalence $h:\Delta_k\times M^n\to M^n$
and sending it to the homotopy equivalence of $n$-ads
$$H :\Delta_k\times M^n~\longrightarrow~\Delta_k\times M^n$$
defined by $H(u,v)=\bigl(\,u,h(u,v)\,\bigr)$. The definitions imply that
the induced maps of homotopy groups from
$\pi_k\bigl(\,E_1(M^n)\,\bigr)$ to $\pi_k\left(\cS^{{DIFF}, s}_{\bullet}
(M^n)\,\right)\cong{\bf S}_k^{s}(M^n)$
are precisely the homomorphisms $\Psi$ defined in (\ref{form: map Psi}).

One useful feature of the preceding constructions is that they yield a
reasonably well-behaved system of composition operations on the structure
sets $\cS^{{DIFF}, x}_{\bullet}(M^n)$.

\begin{prop}\label{prop: composition operators}
Let $M^n$ be a closed, smooth simply connected manifold, and let
$\alpha\in \pi_p(S^q)$, where $p\geq 1$.  Then $\alpha$ induces
a map of structure sets
$$C(\alpha):{\bf S}_q^{s}(M^n)~\longrightarrow {\bf S}_p^{s}(M^n)$$
such that there is a commutative diagram
\[
\xymatrix{
\pi_q\bigl(\,E_1(M^n)\,\bigr)\ar[r]^\Psi\ar[d]^{\alpha^*}&
{\bf S}_q^{s}(M^n)\ar[r]\ar[d]^{C(\alpha)}&
\left[\Sigma^q(M^n_+),F/O\right]\ar[d]^{\alpha^*}&\\
\pi_p\bigl(\,E_1(M^n)\,\bigr)\ar[r]^\Psi&
{\bf S}_p^{s}(M^n)\ar[r]&
\left[\Sigma^p(M^n_+),F/O\right]
}
\]
in which the right and left hand maps $\alpha^*$ are defined by composition.
Furthermore, if $p\geq 2$ and $\alpha$ desuspends to $\alpha\in \pi_{p-1}
(S^{q-1})$, then $C(\alpha)$ is a homomorphism of abelian groups.
\end{prop}

\begin{proof}
If $(X,x)$ and $(Y,y)$ are pointed topological spaces and $f:(X,x)\to
(Y,y)$, then composition with $\alpha$ induces mappings
$$\alpha_W^*:\pi_q(W,w)~~\longrightarrow~~\pi_p(W,w)~,\quad
{\rm where}\quad(W,w)~=~(X,x)~{\rm or}~(Y,y)$$
such that $f_*\tinycirc\alpha_X^*=\alpha_Y^*\tinycirc f_*$, and a similar result
for mappings of $\Delta$-sets follows by taking geometric realizations.
Therefore the existence of $C(\alpha)$ follows from the isomorphism
$${\bf S}_\ast^{s}(M^n)~~\cong~~\pi_\ast\left(\,\cS^{{DIFF}, s}_{\bullet}
(M^n)\,\right).$$
To verify the existence of the commutative diagram, note that the horizontal
maps are the induced morphisms of homotopy groups associated to the
$\Delta$-set mappings $\cH_\bullet(M^n)\to\cS^{{DIFF}, s}_{\bullet}
(M^n)$  and $\cS^{{DIFF}, s}_{\bullet}(M^n)\to\cF_{\bullet} (M^n, F/O)$,
so that the commutativity of the diagram follows from the associativity
of composition for functions.
Finally, the additivity statement follows from Remark~
\ref{rmk: composition additivity}.
\end{proof}

We shall need the following simple consequence of
Proposition~\ref{prop: composition operators}:

\begin{cor}
\label{cor: zero compositions}
In the notation of
\textup{Proposition~\ref{prop: composition operators}}, if $\alpha$ desuspends to an
element of $\pi_{p-1}(S^{q-1})$ with order $A$, and the order $B$ of
$y\in {\bf S}_q^{s}(M^n)$ is finite and prime to $A$, then
$C(\alpha)y=0$ in ${\bf S}_p^{s}(M^n)$.
\end{cor}

\begin{proof}
By Proposition~\ref{prop: composition operators}
we know that $C(\alpha)$ is additive, and hence the
order of $C(\alpha)y$ divides $B$.  On the other hand, we also know
that $C(\alpha)y=y_*\alpha$ where $y_*:\pi_p(S^q)\to {\bf S}_p^{s}(M^n)$
is the induced map of homotopy groups. Since $y_*$ is always additive,
it follows that the order of $C(\alpha)y=y_*\alpha$ also divides $A$.
But $A$ and $B$ are relatively prime, so it follows that $C(\alpha)y$
must be zero.
\end{proof}

\section{Stably trivial self-equivalences from $\pi_7(E_1(\cpq ))$ are trivial}
\label{sec: stably trivial is trivial}

Consider the homomorphism
\[\Psi:\pi_7(F_{S^1}(\mathbb C^3))~~\cong~~\pi_7(E_1(\cpthree))~~\longrightarrow~~
{\bf S}_7^{s}(\cptwo)\]
defined in (\ref{form: map Psi}) and the stabilization map
$s_{3\ast}\co\pi_7(F_{S^1}(\mathbb C^{3}))\to \pi_7(F_{S^1}(\mathbb C^{4}))$.
As the title of this section suggests, here is what we want to prove:

\begin{prop}\label{prop: stably trivial A}
If $\alpha\in \pi_7(E_1(\cptwo))\cong\pi_7(F_{S^1}(\mathbb C^3))$ lies in
the kernel of the homomorphism $s_{3\ast}$, then $\Psi(\alpha)=0$.
\end{prop}

\begin{cor}\label{cor: stably trivial B}
Suppose that
$\alpha\in \pi_7(E_1(\cptwo))\cong\pi_7(F_{S^1}(\mathbb C^3))$ lies in
the kernel of $s_{3\ast}$ and is represented by an equivariant
self-equivalence ${\widetilde f}$ of $S^7\times S^{5}$, so that
the homotopy class of ${\widetilde f}$ is given by
$\Gamma\tinycirc\Psi(\alpha)$.  Then the
induced homotopy self-equivalence $f = {\widetilde f}/S^1$ on
the orbit space $S^7\times \cptwo$ is homotopic to a diffeomorphism.
\end{cor}

\begin{proof}[Proof of Proposition~\ref{prop: stably trivial A}]
By the commutativity of the diagram in
Proposition~\ref{prop: composition operators}, we have
$\Psi\,\eta_6^*=C(\eta_6)\,\Psi$, where $\eta_6\in
\pi_7(S^6)\cong\mathbb{Z}_2$ is nontrivial, and if we combine this
with Proposition~\ref{prop: eta composition} we see that $\eta_6^*$ maps onto
the kernel of $s_{3\ast}$.
Therefore $\alpha=\eta_6^*\alpha'$ for some $\alpha'$, so that
$\Psi(\alpha)= C(\eta_6)\Psi(\alpha')$.  The class $\alpha'$ has
finite order because the spectral sequences of \cite{Sch73}
imply that $\pi_6\left(F_{S^1}(\mathbb C^3)\right)$ is finite, and therefore
$\Psi(\alpha')$ also has finite order.  Since $\eta_6$
desuspends to $\eta_5\in\pi_6(S^5)\cong\mathbb{Z}_2$, by
Corollary~\ref{cor: zero compositions} it will suffice to prove
that $C(\eta_6)$ annihilates the Sylow
$2$-subgroup in the torsion subgroup of ${\bf S}_6^{s}(\cptwo)$.

Consider the following terms in the partial surgery exact sequence for $\cptwo$
(as noted in \cite{ABK}, this portion also exists for 4-manifolds):
$$0=L_{11}(1)~\rightarrow~{\bf S}_6^{s}(\cptwo)~
\rightarrow~\left[\Sigma^6\cptwo\vee S^6,F/O\right]\,\cong\,
\left[\Sigma^6\cptwo,F/O\right]\oplus\pi_6(F/O).$$
The map $C(\eta_6)$ is additive by (1.7) of \cite{Tod} and
$\eta_6^*$ is trivial on $\pi_6(F/O)\cong\mathbb{Z}_2$
because $\pi_7(F/O)=0$. Thus everything reduces to proving that
the torsion subgroup of $\left[\Sigma^6\cptwo,F/O\right]$ has odd order.

Since $\cptwo$ is the mapping cone of the Hopf map, which
represent $\eta_2\in \pi_3(S^2)$, the group $[\Sigma^6\cptwo, F/O]$
fits into the following commutative diagram, whose
rows are exact cofiber sequences for
the map $\Sigma^6\cptwo\to \Sigma^6(\cptwo/\cpone)=S^{10}$,
and columns are portions of
the exact homotopy sequence of the fibration $p\co F\to F/O$.
\[
\xymatrix{
& & &\pi_8(O)\ar[r]^{\eta^\ast\ \ }\ar[d]&
\pi_{9}(O)\ar[d]\\
\pi_{9}(F)\ar[r]^{\eta^\ast\ \ }\ar[d]&  \pi_{10}(F)\ar[r]\ar[d]&
[\Sigma^6\cptwo, F]\ar[r]\ar[d]&\pi_8(F)\ar[r]^{\eta^\ast\ \ }\ar[d]&\pi_{9}(F)\ar[d]\\
\pi_{9}(F/O)\ar[r]^{\eta^\ast\ \ }& \pi_{10}(F/O)\ar[r]&
[\Sigma^6\cptwo, F/O]\ar[r]&\pi_8(F/O)\ar[r]^{\eta^\ast\ \ }\ar[d]&
\pi_{9}(F/O)\ar[d]\\
& & &\pi_8(BO)\ar[r]^{\eta^\ast\ \ }&
\pi_{9}(BO)
}
\]
The results of Toda~\cite{Tod} and Adams~\cite{Ada-J-IV} yield the
following information about mappings in the diagram:

\begin{enumerate}
\item The group $\pi_{10}(F)$ is isomorphic to $\mathbb{Z}_2\oplus
\mathbb{Z}_3$ and the image of $\eta^*$ in this group has order
2~\cite[Chapter XIV]{Tod}.
\item  The mapping $\eta^*$ on $\pi_8(F)$ is injective~\cite[Chapter XIV]{Tod}.
\item  If $k=8$ or 9 then the map from $\pi_k(O)\cong\mathbb{Z}_2$ to
$\pi_k(F)$ is injective~\cite[discussion following Theorem 1.2]{Ada-J-IV}.
\item  The kernel of the
mapping $\eta*$ on $\pi_8(BO)\cong\mathbb{Z}$ has index
2~\cite[discussion following Theorem 1.2]{Ada-J-IV}.
\end{enumerate}
Diagram chases now imply that 
\begin{itemize}
\item[\textup{(i)}]
the torsion subgroup $T$ of
$[\Sigma^6\cptwo, F/O]$ maps to a subgroup of $\pi_8(F/O)$
in the image of $\pi_8(F)\to\pi_8(F/O)$, 
\item[\textup{(ii)}] the restriction
of $\eta^*$ to the torsion
subgroup of $\pi_8(F/O)$ is injective.  
\end{itemize}
Therefore $T$ must be
contained in the image of $\pi_{10}(F/O)\to [\Sigma^6\cptwo, F/O]$.
Since $\pi_{10}(F/O)\cong \mathbb{Z}_2\oplus
\mathbb{Z}_3$ and the image of $\eta^*$ in this subgroup has
order 2, it follows that $T\cong\mathbb{Z}_3$, and as noted
above this yields the proposition.
\end{proof}

\begin{proof}[Proof of Corollary~\ref{cor: stably trivial B}]
This follows trivially because the extension by diffeomorphism
mapping 
$$\Gamma\co {\bf S}_7^s(\cptwo)\to {\bf S}^s(S^7\times \cptwo),$$
see (\ref{form: map Gamma}),
is base point preserving, where in each case the base point is represented
by the identity map of the given manifold.
\end{proof}

\section{The normal invariant of the stable element of order $2$}
\label{sec: order 2 nontrivial normal inv}

As mentioned in Section~\ref{sec: spectral seq},
the group $\pi_7(F_{S^1})$ has a unique order two
element, and in this section we show that the corresponding homotopy
self-equivalences $f$ of $S^7\times \cpq$ (where $q\ge 3$) have
nontrivial normal invariants; if $q=2$ the same argument is also
valid for the nontrivial class in $E^\infty_{1,6}(G\mathbb C^{3})=
E^2_{1,6}(G\mathbb C^{3})=\pi_{11}(S^5)
\cong\mathbb Z_2$ (see Proposition~\ref{prop: bottom filtration}).
A proof of this was given in~\cite{Sch87}, but since it depends upon a
result whose proof has not yet been published we shall give a
complete derivation here by somewhat different methods.

Proposition~\ref{prop: tangentials} shows that
$f$ is tangential. Let $\g$ denote the normal bundle of
$X:=S^7\times\cpq$ in some higher dimensional $S^{m+2q+7}$,
let $\hat f$ denote an arbitrary self-map of $\g$ that covers $f$,
let $T(\hat f)$ be the induced self-map of its Thom space $T(\g)$, and
let $q\co S^N\to T(\g )$ be the map that collapses
the complement of a tubular neighborhood of $X$ to a point.  We shall use
these mappings to describe the normal invariant of $f$ as follows:

\begin{claim}\label{claim: sdual}
In the notation of the preceding paragraph, a lifting of
the (smooth) normal invariant ${\mathfrak q}(f)$ to $[X, F]$
is given by the Spanier-Whitehead (or $S$-) dual of $T(\hat f)\tinycirc q$, where
the set $[X,F]$ of free homotopy classes is identified with
the set of based homotopy classes $[X_+, F]=\{X_+,S^0\}$.
\end{claim}
\begin{proof}
This is well known, but references in the literature are slightly elusive.
For topological surgery theory, the analogous  result  is worked out clearly
and explicitly in Section 2 of \cite{MTW}, and the discussion in that paper only
uses standard formal properties of topological manifolds and their stable normal
bundles.  All of these properties also hold in the category of smooth manifolds,
and accordingly the
approach in \cite{MTW} also works for smooth surgery theory in a setting like
that of Section II.4 in~\cite{Bro-book}.
\end{proof}

The proof of Proposition~\ref{prop: unstable homotopy group} shows that $f$
comes from the $E^2_{1,6}$-term which is in the bottom filtration,
which means that it comes from
$E^1_{1,6}=\pi_7({\rm Filt}^{(1)}, {\rm Filt}^{(0)})$,
where ${\rm Filt}^{(0)}=\{\bf{id}\}$ and
${\rm Filt}^{(1)}$ consists of functions in $F_{S^1}(\mathbb C^{q+1})$
that restrict to inclusions on the standard
subsphere $S^{2q-1}\subset S^{2q+1}$.

Changing $f$ within its homotopy class, we can assume that
the restriction of $f$ to the submanifolds
$S^7\times\mathbb {CP}^{q-1}$ and $\{\ast\}\times\cpq$
are the standard inclusions, and by the
Homotopy Extension Property we may also arrange $f$
to equal the identity on a smooth regular neighborhood of
the union of these submanifolds.
The complement of the interior of the regular neighborhood
is an smoothly embedded $(2q+7)$-disk, which we denote $D$.

We shall need the following homotopy-theoretic fact:

\begin{fact}\label{fact: pinch}
Let $M^n$ be an $n$-manifold, let $D\subset M^n$ be a closed
bicollared coordinate disk, let $M_0:=M-\mathrm{Int}(D)$, and
let $j$ be a continuous map of $M$ into a manifold $Y$.
Given a continuous map $\a\co S^n\to Y$,
let $\widehat\a$ be the composition
\[
\xymatrix{
M^n~\ar[r]^{\!\!\!\!\!\!\!\!\!\!\!\mathrm{pinch}}&
\ M^n\vee S^{n}\ \ar[r]^{\:\:\:\:\:\:{j}\vee {\a}\ }&
\ Y\vee Y \ar[r]^{\:\:\:\:\:\mathrm{fold}}&
Y
}
\]
where the {\rm pinch} map collapses $\d D$ to a point. Then
every continuous map $M\to Y$ that
equals $j$ on a neighborhood of $M_0$ is homotopic to
$\widehat\a$ for some $\a\co S^n\to Y$.
\end{fact}
\begin{proof}
This follows by applying~\cite[III.6.21]{Wh} to the cofiber sequence
$S^{n-1}=\partial M_0\to M_0\to M^n$.  A relatively minor misprint in
\cite{Wh} seems worth mentioning in order to avoid confusion:  The operation
of $\pi_n(Y)$ on $[M,Y]$ in~\cite[p. 136]{Wh} is induced by a coaction map
$M\to M\vee S^n$ instead of
a map $M\to M\wedge S^n$ as stated there (the codomain of the map $\theta$ on
line 4 is the wedge and not the smash product).
\end{proof}

The next goal is to apply~\ref{fact: pinch} and obtain a factorization of
$f$ which will be used to compute its normal invariant.

If $p_1$ and $p_2$ are the coordinate projections
from $X=S^7\times\cpq$ to $S^7$ and $\cpq$ respectively, then
the compositions $p_1\tinycirc f$ and $p_2\tinycirc f$
agree with $p_1$ and $p_2$ off $D$.
Applying Fact~\ref{fact: pinch} with $j$ equal to
$p_1$ and $p_2$, we obtain homotopy classes
$[\a_1]\in \pi_{2q+7}(S^7)$ and $[\a_2]\in \pi_{2q+7}(\cpq)$
such that $\widehat{\a}_i$ is homotopic to $p_i\tinycirc f$.
Since maps into a product are determined by maps into their factors,
this yields a self-map $\widehat\a=({\widehat{\alpha}}_1, {\widehat{\alpha}}_2)$
of $X$ which is homotopic to $(p_1\tinycirc f, p_2\tinycirc f)=f$.

In fact, $p_1=p_1\tinycirc f$  because
$f$ comes from $\pi_7(E_1(\cpq))$, so we can choose the class
$\alpha_1$ in Fact~\ref{fact: pinch}
to be constant.  On the other hand,
$\a_2$ cannot be nullhomotopic because
$f$ is not homotopic to $\mathbf{id}(X)$.
The map $\a_2$ lifts to $S^{2q+1}$ because the orbit space
projection $\omega\co S^{2q+1}\to\cpq$ induces isomorphisms
$\pi_k(S^{2q+1})\to\pi_k(\cpq)$ for all $k\geq 3$.
If $q\ge 2$, then $\pi_{2q+7}(S^{2q+1})\cong\mathbb{Z}_2$, and
if $[g]$ be the nontrivial element of $\pi_{2q+7}(S^{2q+1})$
then we must have $\omega_*([g])=[\a_2]$.

In summary, $\a$
is homotopic to ${\bf SL}\tinycirc \omega\tinycirc g$,
where
{\bf SL} denotes the slice inclusion of $\cpq$ in $S^7\times \cpq$,
and therefore
the map $f$ is homotopic to the composite in the diagram below:
\[
\xymatrix{
X\ar[r]^{\!\!\!\!\!\!\!\!\!\!\!\!\!\!\!\!\!\!\!\!\mathrm{pinch}}&
\ X\vee S^{2q+7}\ \ar[r]^{{\bf id}\vee g\ }&
\ X\vee S^{2q+1}\ \ar[r]^{\:\:\:{\bf id}\vee \omega\ }&
\ X\vee \cpq\ \ar[r]^{\:\:\:\:\:{\bf id}\vee{\bf SL}\ }&
\ X\vee X\ \ar[r]^{\:\:\:\:\:\:\:\mathrm{fold}}& X
}
\]
Since {\bf SL} is an embedding of a proper subspace,
the composite
${\bf SL}\tinycirc\omega\tinycirc g$ actually
determines a class in $\pi_{2q+7}\left(X-\{p\}\right)$
for some $p\in X$ not in the image of {\bf SL}.

The study of this construction, which {\sl twists the identity map of a manifold
$M^n$ by a class in $\pi_n(M^n-\{p\})$\/}, dates back to
\cite[Section 7]{nov}, particularly in the case where the stable normal bundle of $M^n$
pulls back trivially under the composite $S^n\to M^n-\{p\}\subset M^n$.  Note that
the latter condition holds for the displayed mapping because the pullback of a
complex vector bundle under the quotient map  $\omega$ is always stably trivial
(since ${\widetilde{K}}(S^{2q+1}) = 0$).
Our computatons of normal invariants in this section will use
the explicit trivialization for this example in
\cite[(8.7)--(8.8)]{BecSch}.

If $\alpha\in \pi_n(M^n-\{p\})$ satisfies the pullback condition in the preceding
paragraph, then as in \cite{nov} a choice of pullback trivialization yields a
twisted suspension $\alpha^*$ in $\pi_{n+k}\left(\,
(M-\{p\})^\nu\,\right)$, where $\nu$ denotes the stable normal
bundle of $M^n$ in $\mathbb{R}^{n+k}$ for some sufficiently large value of
$k$, and the normal invariant of the twisting of ${\bf id}_M$ by $\alpha$ is determined
by the $S$-dual of $\alpha^*$, viewed as an element of the
group $\{M^n,S^0\}$.

The same methods also yield the following refinement:  {\sl If
$W\subset M$ is a codimension zero submanifold of $M-\{p\}$ with boundary
and $\alpha$ lies in the image of $\pi_n(W)$, then one can lift the twisted
suspension to a class $\beta^*\in \pi_{n+k}(W^\nu)$ and the normal invariant
is the image of the $S$-dual of $\beta^*$, which is a class in
$\{V/\d V,S^0\}$, under the homomorphism
$c^*:\{V/\d V,S^0\}\to \{M^n,S^0\}$ induced by
the collapsing map $c:M\to V/\d V$.}

We are now ready to prove the main computational result of this section:
{\sl If $q\geq 2$ and $f$ is a homotopy self-equivalence of $X=S^7\times \cpq$
which is not homotopic to the identity and is in the lowest
nontrivial filtration of
$\pi_7(E_1(\cptwo))$ with respect to the spectral sequence of
Section~\ref{sec: spectral seq}, then the normal invariant of $f$ is
nontrivial.\/}\ (In fact, our results describe the restriction of the normal
invariant to $S^7\times\cpone$ explicitly).

By the preceding discussion, the normal invariant is the $S$-dual of
the composite
${\bf SL}^{\rm twisted}\tinycirc \omega^{\rm twisted}\tinycirc \nu^2$
where $\nu^2$ is the nontrivial element of $\pi_{2q+7}(S^{2q+1})\cong
\pi_6^{\bf S}\cong\mathbb Z_2$,
$\omega^{\rm twisted}\in\pi_{2q-1}\left(\,(\cpq)^\xi\right)$ is the twisted
suspension map associated to the previously described trivialization
for the pullback of the stable normal bundle $\xi$ of $\cpq$ by the
quotient map $\omega:S^{2q+1}\to\cpq$,
and ${\bf SL}^{\rm twisted}$ is the map from $\left(\cpq\right)^\xi$
to $S^k_+\wedge\left(\cpq\right)^\xi$ induced by the slice inclusion of
$\cpq$ in $S^7\times \cpq$.  Thus the proof of the main result reduces to
analyzing the $S$-dual of this composite $S$-map.

The $S$-dual of the explicitly chosen twisted suspension
$\omega^{\rm twisted}$ is the Umkehr map $p^{\,!}:\Sigma (\cpq_+)\to S^0$
(see~\cite[Section 8, especially pp. 17--18]{BecSch}), and the
$S$-dual of
the twisted slice inclusion is the map
$S^{-k}_+\wedge\cpq\to\cpq$ induced by the $S$-map $S^{-k}_+=
S^{-k}\vee S^0\to S^0$ which is trivial on the first summand and
the identity on the second.  Since
$\nu^2$ is self-dual, it follows that the normal invariant
is determined by the composite $\Lambda$ of the Umkehr map $p^{\,!}$ and
the smash product of the suspension of
$\nu^2$ (viewed as an $S$-map from $S^7$ to $S^1$) with the identity
on $\cpq_+$.  Thus everything
reduces to computing $\Lambda|_{S^{7}_+\wedge\cpone}$ when $q\geq 2$.

As a first step, we show that the restriction of the Umkehr map to
$\Sigma(\cpone)$ is a generator of
$\{\Sigma(\cpone), S^0\}\cong
\pi_3(F)\cong\pi_3^{\bf S}\cong\mathbb Z_{24}$.
Indeed, the inclusion $S^3=\Sigma(\cpone)\to
\Sigma(\mathbb{CP}^\infty_+)$ represents a generator of
the $\mathbb Z$-factor in
\[
\pi_3^{\bf S}(\Sigma(\mathbb{CP}^\infty_+))=
\pi_3^{\bf S}(\Sigma(\mathbb{CP}^\infty)\vee S^1))=
\pi_3^{\bf S}(\Sigma(\mathbb{CP}^\infty))\oplus
\pi_3^{\bf S}(S^1)\cong\mathbb Z\oplus\mathbb Z_2.
\]
By the commutativity property of the diagram~\cite[(6.10)]{BecSch}
the homomorphism $\pi_3^{\bf S}(\Sigma(\mathbb{CP}^\infty_+))\to \pi_3^{\bf S}$
induced by the Umkehr map $p^{\,!}$ coincides with the forgetful map
$\pi_3(F_{S^1})\to \pi_3(F)$, and by~\cite[Theorem 11.1]{BecSch}
the image of $\pi_3(U)\to \pi_3(F_{S^1})$
induced by inclusion is an infinite cyclic summand. Thus the image
of $\pi_3^{\bf S}\left(\Sigma(\cpone)\right)$ is equal to the image of
a generator for $\pi_3(U)=\mathbb Z$
up to an element of order $\le 2$. Since the canonical homomorphism
from $\pi_3(U)$ to $\pi_3(O)$ is onto ({\it cf.\/} \cite{BotMil})
and similarly for $\pi_3(O)\to \pi_3(F)$ ({\it e.g.\/}, by
\cite{Ada-J-IV}), it
follows that either infinite cyclic generator must map to a
generator of $\pi_3(F)$.

The discussion in the preceding paragraph implies that the restriction of
$\Lambda=p^{\,!}\tinycirc\left(\Sigma\nu^2\wedge\mathbf{id}(\cpq_+)\,\right)$ to
$S^7\wedge\cpone\cong S^9$ is the composition of
$\Lambda|_{\Sigma\cpone}$, which we have shown to be a generator of
$\pi_3^{\bf S}\cong\mathbb Z_{24}$, and $\Sigma^3\nu^2\in
\{S^9,S^3\}\cong\pi_6^{\bf S}$.  Now $\pi_3^{\bf S}$ is generated
by $\nu$, and the third composition power
$\nu^3$ is nonzero in $\pi_9^{\bf S}=(\mathbb Z_2)^3$ \cite[Chapter XIV]{Tod}.
Since all nontrivial elements of the latter group have order $2$ and two
generators of $\mathbb Z_{24}$ differ by an even number mod 24, it follows that
the composition of $\nu^2$ with any generator of $\pi_3^{\bf S}$
is $\nu^3$, and this identifies $\Lambda|_{S^9}$ as an explicit nontrivial
element of $\pi_9^{\bf S}$.   Thus ${\mathfrak q}(f)$ restricted to
$\Sigma^7\cpone$, which is the image of $\Lambda|_{\Sigma^7\cpone}$ under the
canonical map from $\{\Sigma^7\cpone,S^0\}\cong[\Sigma^7\cpone,F]$ to
$[\Sigma^7\cpone,F/O]$,
is given by the image of $\nu^3$ in $\pi_9(F/O)$.

Finally, we need to verify that ${\mathfrak q}(f)|_{\Sigma^7\cpone}$ is nonzero.
The class $\Lambda|_{\Sigma^7\cpone}$ maps to ${\mathfrak q}(f)|_{\Sigma^7\cpone}$
via $\pi_9^{\bf S}\cong \pi_9(F)\to \pi_9(F/O)$, so it remains to show that
$\nu^3$ is not in the kernel of $\pi_9(F)\to \pi_9(F/O)$,
which equals the image of the $J$-homomorphism
$\pi_9(O)\to\pi_9(F)$.  For the sake of completeness, we indicate how
this can be verified.

Note that by~\cite[Theorem 1.2]{Ada-J-IV} the square of the Hopf map
$\eta^2$ induces a nonzero homomorphism
$\mathbb Z=\pi_{7}(O)\to\pi_{9}(O)=\mathbb Z_2$.
Thus if $\g$ denotes a generator of $\pi_{7}(O)$, then
$\g\eta^2$ generates $\pi_{9}(O)$. The $J$-homomorphism
$\pi_{7}(O)\to \pi_{7}(F)\cong\mathbb Z_{240}$
is onto, so $\pi_7(F)$ is generated by $J(\g)$ (which is
$\sigma$ in \cite{Tod}), and it follows
that the image of $J\co\pi_{9}(O)\to \pi_{9}(F)\cong
\mathbb Z_2\oplus\mathbb Z_2\oplus\mathbb Z_2$
is generated by $J(\g\eta^2)=\sigma\eta^2=\eta^2\sigma$.
Tables in~\cite[pp. 189--190]{Tod} imply that the image of
$\eta^2\co\pi_7(F)\to\pi_9(F)$ is generated by
$\eta^2\sigma=\nu^3+\eta\tinycirc\e$.  Therefore
$\nu^3$ is not in the image of the $J$-homomorphism.

\section{Dichotomy principles and skeletal filtrations}
\label{sec: dichotomy proof and skeleton filtrations}

We start by proving an important particular case of
Theorem~\ref{thm: main dichotomy}; of course, only
one direction is nontrivial.

\begin{prop}
\label{prop: dichotomy for pi7-monoid)}
For $q\ge 2$ let $f$ be a homotopy self-equivalence of
$S^7\times\cpq$ that comes from an element of
$\pi_7\left(\,F_{S^1}(\mathbb C^{q+1})\,\right)$.  Then $f$ is
homotopic to a diffeomorphism if and only if $f$ has trivial
normal invariant.
\end{prop}
\begin{proof}
Suppose first that $q\ge 3$ so that
$\pi_7(F_{S^1}(\mathbb C^{q+1}))\cong \pi_7(U_{q+1})\oplus \mathbb Z_2$.
We write $f={f_1}{\tinycirc}{f_2}$ where $f_1$
comes from the $\pi_7(U_{q+1})$-factor and $f_2$ comes from
the $\mathbb Z_2$-factor.
Then $f_1$ is homotopic to a diffeomorphism,
because $U_{q+1}$ acts on $\cptwom$ by
diffeomorphisms.  By the discussion in
Section~\ref{sec: order 2 nontrivial normal inv}
either $f_2$ is homotopic to identity, or $f_2$ has nontrivial
normal invariant, and the claim follows.

Suppose now that $q=2$.
By Proposition~\ref{prop: unstable homotopy group} the map
$f$ is homotopic to the composition ${f_2}{\tinycirc}{f_1}$
of homotopy self-equivalence $f_1$, $f_2$, where each factor
has order at most $2$, the map $f_1$ comes from an element
in the kernel of $\pi_7(F_{S^1}(\mathbb C^3))\to\pi_7(F_{S^1})$,
and $f_2$ is either homotopic to identity, or else
comes from an element that is mapped to the
unique order $2$ element of $\pi_7(F_{S^1})$.
Corollary~\ref{cor: stably trivial B} implies that
$f_1$ is homotopic to a diffeomorphism.
If $f_2$ is homotopic to the identity we are done, so
suppose that $f_2$ is not homotopic to the identity.  In the
latter case the proof in
Section~\ref{sec: order 2 nontrivial normal inv} shows that
the restriction of ${\mathfrak q}(f_2)$ to $\Sigma^7(\cpone)
\vee S^7$ is nontrivial, and since $f_1$ is homotopic to a
diffeomorphism the same also holds for
${\mathfrak q}(f_2\tinycirc f_1 = f)$.
\end{proof}

The next step is to prove a Dichotomy Property for
normal invariants
of homotopy self-equivalences coming from maps
$\cpq\to E_1(S^k)=SG_{k+1}$.

\begin{prop}
\label{prop: dichotomy for maps in Gk}
Let $X$ be a closed connected smooth $n$-manifold, let
$k\geq 2$ with $n+k\geq 5$, let $u:X\to SG_{k+1}$
be continuous, and let $f\co S^k\times X\to S^k\times X$
denote the homotopy self-equivalence arising from $u$.
Then either $f$ is homotopic to a diffeomorphism, or else
$f$ is not normally cobordant to the identity.  In the
first case, the diffeomorphism extends to a diffeomorphism
of $D^{k+1}\times X$.\end{prop}

\begin{proof}  A key point is that
every homotopy self-equivalence of $S^k$ extends to
$D^{k+1}$ by the cone construction, which implies that
$f$ extends to a  homotopy self-equivalence $\bar f$ of
$D^{k+1}\times X$ and hence yields a homotopy structure
on $D^{k+1}\times X$.  By Wall's $\pi-\pi$
Theorem~\cite[Chapter 3]{Wal-book}, the map $\bar f$ is homotopic to a
diffeomorphism if and only if its normal invariant is
trivial.  The restriction map
\begin{equation}
\label{form: restriction}
[X,F/O]\cong [D^{k+1}\times X,F/O]~\longrightarrow~
[S^k\times X,F/O]
\end{equation}
is split injective, {\it i.e.\/}
if $G\co D^{k+1}\times X\to F/O$ restricts to
$g\co S^k \times X\to F/O$, then $g\vert_{\{\ast\}\times X}$
corresponds to $G$ under $[D^{k+1}\times X,F/O]\cong [X,F/O]$.
By the geometric definition of normal invariant, ${\mathfrak q}(\bar f)$
maps to ${\mathfrak q}(f)$ via
restriction to the boundary, and therefore, by the previous sentence,
${\mathfrak q}(f)$ maps to ${\mathfrak q}(\bar f)$
by restriction to ${\{\ast\}\times X}$.
It follows that ${\mathfrak q}(f)$ is trivial if and only if
${\mathfrak q}(\bar f)$ is trivial.
If it is trivial, then Wall's $\pi-\pi$ Theorem
implies that $\bar f$, and hence $f$ is homotopic to a diffeomorphism.
\end{proof}

One step in the preceding argument is important enough
to be stated explicitly: The normal invariant of $f$
lies in the image of $[X,F/O]$ in
$[S^k\times X, F/O]$ under the restriction map (\ref{form: restriction}).
We shall need a strengthened
form of this result.

\begin{cor}
\label{cor: normal invariant rel A}
Under the assumptions of
Proposition~\ref{prop: dichotomy for maps in Gk}, if $A$ is a
subcomplex in some triangulation of $X$ and if
the restriction $u\vert A$
is trivial in $[A,SG_{n+1}]$, then the restrictions
of ${\mathfrak q}(f)$ to $A$ and $S^k\times A$ are
also trivial.\end{cor}

\begin{proof}
If $B$ is a closed regular neighborhood of $A$, then
by the Homotopy Extension Property
we may replace $u$ with some $v$ in the same homotopy
class such that the restriction of $v$ to $B$ is
constant with value ${\bf id}(X)$.  Let $F$ be the homotopy
self-equivalence of $S^k\times X$ that corresponds to $v$.
Then $F$ maps $S^k\times B$ to itself by the identity.
Set $M_2:=S^k\times B$ and $M_1:= X\setminus\mathrm{Int}(M_2)$.
By Proposition~\ref{prop: restriction of normal invariants}
the collapsing map $X\to X/M_2$ takes
${\mathfrak q}(F\vert {M_1})$ to ${\mathfrak q}(F)$.
By exactness of the cofiber sequence
$[X/M_2,F/O]\to [X,F/O]\to [M_2,F/O]$
the restriction of ${\mathfrak q}(F)$ to $M_2$ is trivial,
and hence the same holds for every subspace of $M_2$,
such as $A$ and $S^k\times A$.
\end{proof}

Now that we have the Dichotomy Property for homotopy self-equivalences
coming from elements in $\pi_7(E_1(\cpq))$ and $[\cpq, E_1(S^7)]$
we only need to see whether their normal invariants can cancel,
so that the normal invariant of the composition
of these homotopy self-equivalences cannot be trivial unless
both summands vanish. This matter is naturally treated
in the framework of skeletal filtrations.

Let ${\bf T}$ be a contravariant functor defined from
the homotopy category of pointed finite cell complexes to
the category
of abelian groups.  If $X$ is a pointed finite cell complex,
then we say that
a class $u\in {\bf T}(X)$ {\sl has skeletal
filtration} $\geq k$,  if the restriction of $u$ to the
$k$-skeleton $X_k$ is trivial, and we say that the
{\it skeletal filtration of $u$ equals} $k$ if $u$
has filtration $\geq k$ but does not have
filtration $\geq k+1$.
The Cellular Approximation Theorem for continuous
maps of $CW$-complexes implies
that the skeletal filtration of a class in ${\bf T}(X)$ does not
depend upon the choice of cell decomposition; in fact, it
follows that the sets ${\bf T}^{\langle k\rangle}(X)$ of elements
with skeletal filtration $\geq k$ are subgroups and define
a filtration of ${\bf T}$ by subfunctors.

\begin{prop}
\label{prop: even filtration}
Suppose that $f$ is a homotopy self-equivalence
of $S^7\times \cpq$ with $q\geq 2$, which comes from an
element of $[\cpq,SG_8]$.  If the normal invariant
of $f$ is nontrivial, then its filtration is an even
number.
\end{prop}

\begin{proof} In view of the proof of
Proposition~\ref{prop: dichotomy for maps in Gk},
we might as well consider the normal invariant for the
homotopy self-equivalence of $D^8\times \cpq$ extended
via the cone construction, which lies in
$[D^8\times\cpq , F/O) ]\cong [\cpq, F/O]$.
Since $\cpq$ has cells only in even dimensions, it
follows that the filtration of a nontrivial element
must be even.
\end{proof}

\begin{rmk}
\label{rmk: existence of self-eq with stably nontrivial norm inv}
On the other hand, if $f$ is a homotopy self-equivalence
of $S^7\times \cpq$ with $q\geq 2$ that comes from an
element of $\pi_7(E_1(\cpq))$, and if ${\mathfrak q}(f)$ is nontrivial,
then the skeletal filtration of ${\mathfrak q}(f)$ is odd.
Indeed, as in the proof of Proposition~\ref{prop: factorization}
we may assume that $f$ is the identity on $S^7\vee \cpq$; hence
${\mathfrak q}(f)$ can be thought of
as an element of $[S^7\wedge \cpq, F/O]$.
Thus if ${\mathfrak q}(f)$ is nontrivial,
then the filtration of ${\mathfrak q}(f)$ is odd because
$S^7\wedge\cpq$ has a cell decomposition
(inherited from the product of the standard
cell decomposition of $\cpq$ and $S^k=D^k\cup D^0$)
whose positive dimensional cells
only appear in odd dimensions from $9$ to $2q+7$.
In fact, the $9$-skeleton of $S^7\wedge \cpq$ is
$S^7\wedge\cpone=S^9$, and it was shown in
Section~\ref{sec: order 2 nontrivial normal inv} that
the restriction of ${\mathfrak q}(f)$ to $S^7\wedge\cpone$
defines a nontrivial element of $\pi_9(F/O)$,
so the filtration of ${\mathfrak q}(f)$ is $9$.
\end{rmk}

\begin{proof}[Proof of Theorem~\ref{thm: main dichotomy}]
By Propositions~\ref{prop: product cpq times sk},
\ref{prop: factorization}, a homotopy self-equivalence $f$
of $S^7\times \cptwo$ is homotopic to
${f_1}\tinycirc {f_2}\tinycirc \phi$
where $\phi$ is a diffeomorphism
and $f_1$, $f_2$ are homotopy self-equivalences coming
from elements in $\pi_7(E_1(\cpq))$, $[\cpq, E_1(S^7)]$,
respectively.
The composition formula for normal invariants says that
\[
{\mathfrak q}(f)={\mathfrak q}({f_1}\tinycirc f_2)=
{\mathfrak q}(f_1)+(f_1)^{\ast -1}{\mathfrak q}(f_2).
\]
By Propositions~\ref{prop: dichotomy for maps in Gk}--\ref{prop: even filtration}
either $f_1$ is homotopic to a diffeomorphism, or
else the filtration of ${\mathfrak q}(f_1)$ is odd.
Similarly, by Proposition~\ref{prop: dichotomy for pi7-monoid)}
either $f_2$ is homotopic to
a diffeomorphism, or else its filtration is
even; general considerations then imply the same conclusion for
$(f_1^*)^{-1}{\mathfrak q}(f_2)$.
Therefore, if ${\mathfrak q}(f)$ is trivial, then
both ${\mathfrak q}(f_1)$ and ${\mathfrak q}(f_2)$ are
trivial, and hence $f_1$, $f_2$ are homotopic to diffeomorphisms,
so $f$ is homotopic to a diffeomorphism.
\end{proof}

\section{The homotopy inertia group of $S^7\times\cptwo$}
\label{sec: index 4}

Here we obtain an optimal strengthening of Taylor's
result (Theorem~\ref{thm: taylor}) for $M=S^7\times\cptwo$.

\begin{thm}
\label{thm: index 4 in inertia}
The subgroup $I_h(S^7\times\cptwo)\cap bP_{12}$ has
index $4$ in $bP_{12}$. The manifolds
$\Sigma(d)\times\cptwo$
fall into $3$ diffeomorphism types, and
$4$ oriented diffeomorphism types.
\end{thm}

It follows from Corollary~\ref{cor: diffeo criterionfor sigma7(d)}
and Remark~\ref{rmk: non-orient diffeo} that
the first sentence in Theorem~\ref{thm: index 4 in inertia}
implies the second one, and this proves
Theorem~\ref{intro-thm: CP2 cancells}.

The inclusion $4\cdot bP_{4r}\subset I_h(S^7\times\cptwo)
\cap bP_{4r}$ is a general phenomenon arising from
the product formula for the surgery obstruction and certain
numerical properties of orders of groups $bP_{4r}$
which we denote $|bP_{4r}|$.
The following lemma generalizes an argument
in~\cite[(6.5)]{Bro-surv} given for $m=1$.

\begin{lem}
\label{lem: inertia vs divisible by 3}
If $m$ is not divisible by $3$, then
$4\cdot\Sigma^{4m+7}(1)=\Sigma^{4m+7}(4)\in
I_h(S^7\times\cptwom)$, and
the manifolds $\Sigma^{4m+7}(d)\times\cptwom$
fall into at most $3$ diffeomorphism types, and at most
$4$ oriented diffeomorphism types.
\end{lem}

\begin{proof}
In the notation of Fact~\ref{fact: browder},
if $d=|bP_{4r}|$, then $h$ is homotopic to a diffeomorphism, so
the group $\mathrm{ker}(\Delta)\subset\mathbb Z$
contains subgroups of indices $|bP_{4r}|$ and $|bP_{4m+4r}|$.

By~\cite{KerMil} and~\cite{Qui}, for $r\ge 2$ the order of $bP_{4r}$
is $a_r2^{2r-2}(2^{2r-1}-1)n_r$
where $a_r$ is 2 if $r$ is odd and 1 if $r$ is even, and $n_r$ is the numerator of $B_r/4r$
where $B_r$ is the corresponding Bernoulli number.
Basic results in number theory imply that
either $n_r=1$ or $n_r$ is equal to a product of
irregular primes.

It is straightforward to check that $7$ divides $|bP_{4r}|$
if and only if $3$ divides $r-2$ (the point is that
$7$ does not divide $n_r$ because the smallest irregular prime
is $37$,
and hence we need to see when $7$ divides $(2^{2r-1}-1)\cdot 2$;
setting $r=3s+u$ with $u\in\{0,1,2\}$, $s\in\mathbb Z$, we get
$2^{2r}-2=8^{2s}\cdot 2^{2u}-2$
which is equal to $2^{2u}-2$ mod $7$, so $u$ must be $2$).

Now specialize to the case $r=2$, for which $|bP_8|=28=4\cdot 7$,
let $bP_{4m+8}=bP_{4r}$ for $r=m+2$, and suppose that
$3$ does not divide $m$, so that $7$ does not divide $|bP_{4m+8}|$.
Since $\mathrm{ker}(\Delta)\subset\mathbb Z$
contains $28\mathbb Z$ and $|bP_{4m+8}|\mathbb Z$,
the group $\mathrm{ker}(\Delta)$ must contain $4\mathbb Z$ as claimed.
\end{proof}

To show that $I_h(S^7\times\cptwo)\cap bP_{4r}$
lies in an index $4$ subgroup we need to
recall a formula of Brumfiel as stated in~\cite[2.2]{Sch87}.
Set $\theta_r:=|bP_{4r}|$; as mentioned above, $\theta_r$
is divisible by 4.
In~\cite[Section 5]{Bru} Brumfiel defines a homomorphism
$f_R\co \Theta_{4r-1}\to \mathbb Z_{\theta_r}$
such that $f_R\vert {bP_{4r}}$
is an isomorphism, and he proves
in~\cite[Proposition II.3, p. 403]{Bru}
that for every closed, oriented, smooth manifold
$N$ of dimension $4r-1\ge 7$, and every homotopy
sphere $\Sigma^{4r-1}\in I_h(N)$ we have the following
equation:
\begin{equation}
\label{form: brumfiel}
f_R(\Sigma^{4r-1})=-\sum_{m=1}^r\left\langle
\frac{\theta_m}{a_m j_m(2m-1)!}\,
L_{r-m}(N)\,p_m(\xi),\, [S^1\times N]
\right\rangle
\end{equation}
Here $a_m$ is 2 if $m$ is odd and 1 if $m$ is even, $j_{m}$ is the order of the
image of the $J$-homomorphism $\pi_{4m}(BO)\to\pi_{4m}(BF)$,
$L_{r-m}$ is the $4(r-m)$-dimensional Hirzebruch polynomial,  $p_m(\xi)$ is the
$m^{\rm th}$ Pontryagin class, $\xi$ is the pullback of some fiber homotopy
trivial vector bundle over the suspension $\Sigma N$ of $N$ under the collapsing
map $S^1\times N\to S^1\wedge N=\Sigma N$, and the sum
is evaluated on the fundamental class of $S^1\times N$.

The proof of~\cite[Theorem 2.1]{Sch87} shows that
each summand in (\ref{form: brumfiel})
has even numerator and odd denominator,
and therefore $f_R(\Sigma^{4r-1})$
is an even multiple of a generator
in the group $\mathbb Z_{\theta_r}$,
which has even order.
One step of that proof
was~\cite[Sublemma 2.3]{Sch87} showing
that $p_m(\xi)$ is divisible by
$j_m$ modulo torsion.
In Lemma~\ref{lem: fiber homot pontr} below we improve this sublemma
by a factor of two when
$N=S^7\times\cptwo$; this will complete the proof of
Theorem~\ref{thm: index 4 in inertia}.

\begin{lem}
\label{lem: fiber homot pontr}
If $\xi$ is a stably fiber homotopically trivial
vector bundle over the suspension of $S^7\times
\cptwo$, then for each positive integer $m$ the
$m^{\rm th}$ Pontryagin class $p_m(\xi)$ is
divisible by $2j_{m}$.
\end{lem}
\begin{proof}
Given two CW complexes $X$ and $Y$ there is a natural homotopy equivalence
between $\Sigma(X\times Y)$ and $\Sigma X\vee\Sigma Y\vee\Sigma (X\wedge Y)$
({\it e.g.\/}, see ~\cite[proof of III.4.6]{Bro-book}). Hence it suffices to
establish the result for bundles over $S^8$, $\Sigma\cptwo$ and $\Sigma^8\cptwo$.
Since $\Sigma\cptwo$ is obtained by attaching
a $5$-cell to $\Sigma\cpone=S^3$ and $\pi_3(BO)=\pi_5(BO)=0$,
we know that $[\Sigma\cptwo, BO]$ is trivial.

A key ingredient in what follows is an integrality result of Bott
(see~\cite{BotMil}), which states that the Pontryagin class $p_m$ of a vector bundle
over $S^{4m}$ is divisible by $a_m\cdot (2m-1)!$, where
$a_m=2$ if $m$ is odd and $a_m=1$ if $m$ is even.

Suppose that $\xi$ is a stably fiber homotopically
trivial vector bundle over $S^8$. Since $\pi_7(F/O)=0$,
the homotopy sequence of the fibration
implies that the J-homomorphism
$\mathbb Z=\pi_8(BO)\to\pi_8(BF)=\mathbb Z_{240}$
is onto, so $\xi$ is stably isomorphic to $240k\omega$
where $k\in\mathbb Z$ and $\omega$ represents a generator in $\pi_8(BO)$.
By Bott's integrality result $p_2(\omega)$ is divisible by $6$.
It follows that $p_2(\xi)=240k p_2(\omega)$, hence
$p_2(\xi)$ is divisible by $240\cdot 6$.  By \cite[Theorem 1.6 and the
subsequent paragraph]{Ada-J-IV} we know that $j_2=240$,
so $p_2(\xi)$ is divisible by $6j_2$, as desired.

Next, suppose that $\xi$ is a
stably fiber homotopically trivial vector bundle over
$\Sigma^8\cptwo$. In the commutative diagram
vertical arrows are $J$-homomorphisms, and
rows are cofiber exact sequence
associated with the mapping cone sequence
$S^2=\cpone\to\cptwo\to\cptwo /\cpone=S^4$.
\[
\xymatrix{
\pi_{11}(BO)\ar[r]\ar[d]& \pi_{12}(BO)
\ar[r]^{\times 2\ \quad}\ar[d]^{\text{onto}}&
[\Sigma^8\cptwo, BO]\ar[r]\ar[d]& \pi_{10}(BO)\ar[r]\ar[d]^{1-1}&
\pi_{11}(BO)\ar[d]\\
\pi_{11}(BF)\ar[r]& \pi_{12}(BF)\ar[r]&
[\Sigma^8\cptwo, BF]\ar[r]& \pi_{10}(BF)\ar[r]&
\pi_{11}(BF)
}
\]
One knows that $[\cptwo, BO]=\mathbb Z$~\cite[Theorem 3.9]{San64},
and by Bott Periodicity $\pi_{12}(BO)=\mathbb Z$,
$\pi_{11}(BO)=0$, $\pi_{10}(BO)=\mathbb Z_2$,
and $[\Sigma^8\cptwo, BO]\cong [\cptwo, BO]$.
Thus the map
\[
[\Sigma^8(\cptwo/\cpone), BO]=\pi_{12}(BO)\to
[\Sigma^8\cptwo, BO]
\]
is multiplication by $\pm\,2$.

Since $J\co \pi_{10}(BO)\to\pi_{10}(BF)$ is
one-to-one~\cite[Theorem 1.3]{Ada-J-IV},
and $\xi$ is stably fiber homotopically trivial, it
follows that $\xi$
is a pullback of some vector bundle $\zeta$ over $S^{12}$.
Since
$J\co \mathbb Z=\pi_{12}(BO)\to\pi_{12}(BF)=\mathbb Z_{504}$
is onto, and $\pi_{11}(BF)=\mathbb Z_6$, a diagram chase
shows that (the class of) $\zeta$ in $\pi_{12}(BO)=\mathbb Z$
lies in $84\mathbb Z$ where $84\cdot 6=504$, so
$\zeta=84\zeta^\prime$ in $\pi_{12}(BO)$.
By Bott's result, $p_3(\zeta^\prime)$ is divisible by $2\cdot 5!$,
so $p_3(\zeta)$ is divisible by $84\cdot 2\cdot 5!$.
Recalling that pullback acts as multiplication by $2$,
we see that $p_3(\xi)$ is divisible by
$2\cdot 84\cdot 2\cdot 5!=80\cdot 504=80j_{3}$,
which completes the proof.
\end{proof}

\section{Nondiffeomorphic codimension $2$ simply connected souls}
\label{sec: codim 2 souls}

In~\cite[Theorem 1.8]{BKS1} the authors showed that if
$S$ and $S^\prime$ are closed simply connected manifolds
of dimension at least 5 such that complex line bundles over
$S$ and $S^\prime$ have diffeomorphic total spaces, then $S^\prime$ is
diffeomorphic to the connected sum of $S$ with a homotopy sphere
which bounds a parallelizable manifold.
We shall prove a partial converse to this statement:

\begin{thm}\label{thm: diffeo total spaces}
Let $\o$ be a nontrivial complex line bundle  over
a closed simply connected $n$-manifold $S$ with $n\ge 5$, and
let $S^\prime$ be the connected sum of $S$ with a homotopy sphere
that bounds a parallelizable manifold.
Let $\o^\prime$ be the pullback of $\o$ via the standard
homeomorphism $S^\prime\to S$. Then the disk bundles
$D(\o^\prime )$, $D(\o )$ are diffeomorphic, except
possibly when $n\equiv 1\ \mathrm{mod}\; 4$ and
$\pi_1(\d D(\omega))$ has even order.
\end{thm}

Before proceeding to prove this result, we shall provide some
insight into its final sentence, proving that under the given
assumptions $\pi_1(\d D(\omega))$ is always a finite cyclic group.

\begin{lem} \label{lem: fund gr of circle bundle}
Let $B$ be a simply connected
closed manifold and let $P\to B$ be the projection of
a nontrivial circle bundle.
Then $\pi_1(P)\cong\mathbb Z_d$,
where the Euler class of the circle bundle is a $d^{\,th}$
multiple of an indivisible
element in the free abelian group $H^2(B)$.
\end{lem}
\begin{proof}
Since $H_1(B)$ is trivial, the Universal Coefficient Theorem implies
that $H^2(B)=\mathrm{Hom}(H_2(B),\mathbb Z)$,
and therefore $H^2(B)$ is free abelian.
Consider the following partial Gysin sequence
\[
0=H^1(B)\to H^1(P)\to H^0(B)\to H^2(B)\to H^2(P)\to H^1(B)=0
\]
in which the middle map is multiplication by the Euler class.
Since the Euler class is nontrivial and $H^2(B)$
has no torsion, we see that the Euler class has infinite order,
so that $\mathbb Z=H^0(B)\to H^2(B)$ must be injective and
$H^1(P)=0$ by exactness.
The Universal Coefficient Theorem implies that
$H_1(P)$ is mapped onto $\mathrm{Hom}(H^1(P),\mathbb Z)=0$
with kernel $\mathrm{Ext}(H^2(P),\mathbb Z)$, which is
isomorphic to the torsion subgroup of $H^2(P)$.
Finally, the homotopy sequence
of the circle bundle $P\to B$ and triviality of $\pi_1(B)$
imply that $\pi_1(P)$ is cyclic, so that $\pi_1(P)=H_1(P)$.
If the Euler class of the circle bundle is the $d^{\,{\rm th}}$
multiple of some indivisible element in $H^2(B)$, then
the exactness of the Gysin sequence implies that the torsion subgroup
of $H^2(P)$ is isomorphic to $\mathbb Z_d$, and therefore
$\pi_1(P)\cong\mathbb Z_d$.
\end{proof}

\begin{proof}[Proof of Theorem~\ref{thm: diffeo total spaces}]
We may assume $n$ is odd, for in even dimensions there are no exotic
spheres which bound parallelizable manifolds. By surgery theory the
standard homeomorphism $f\co S^\prime\to S$ has trivial normal
invariant in $[S, F/O]$.
If $\bar f\co D(f^\#\omega)\to D(\omega)$ is
the induced map of $2$-disk bundles, and if
$p\co D(\omega)\to S$ denotes
the disk bundle projection, then the normal invariants of
$\bar f$ and $f$ are related as
${\mathfrak q}(\bar f)=p^\ast {\mathfrak q}(f)$ ({\it e.g.\/},
see~\cite[Lemma 5.9]{BKS1}); thus
${\mathfrak q}(\bar f)$ is trivial. If we set $N:=D(\omega)$,
then the element of the structure set represented by $\bar f$
lies in the image of
\[
\xymatrix{
\Delta\co L^s_{n+3}
\left(\,\pi_1(N),\pi_1(\partial N)\,\right)\ar[r]& {\bf S}^s(N)~.
}
\]
Lemma~\ref{lem: fund gr of circle bundle} implies that
$\pi_1(\partial N)=\pi_1(S(\omega))\cong \mathbb Z_d$ for some $d\ge 1$,
in which case the above relative Wall group
$L^s_{n+3}\left(\,\pi_1(N),\pi_1(\partial N)\,\right)$ can be
(and often is) denoted by $L^s_{n+3}(\mathbb Z_d\to 1)$.

If $\mathbb Z_d=1$, then Wall's $\pi-\pi$ Theorem
implies that  $L^s_{n+3}(\mathbb Z_d\to 1)$ is trivial, so
$\bar f$ is homotopic to a diffeomorphism as desired.
In general, there is a short exact
sequence:
\[
\xymatrix{
L^s_{n+3}(\mathbb Z_d)\ar[r]
&
L^s_{n+3}(1)\ar[r]&
L^s_{n+3}(\mathbb Z_d\to 1)\ar[r]&
L^s_{n+2}(\mathbb Z_d)\ar[r]&
L^s_{n+2}(1)
}
\]
In this sequence the left and right arrows are split
surjections with one-sided inverses induced by the trivial
inclusion $1\to\mathbb Z_d$, and therefore $L^s_{n+3}(\mathbb Z_d\to 1)$
is isomorphic to the kernel of the split surjection
$L^s_{n+2}(\mathbb Z_d)\to L^s_{n+2}(1)$.
It is stated in~\cite[p.\,227]{HamTay} and
proved in~\cite[Sections 10--12]{HamTay}
that $L^s_{n+2}(\mathbb Z_d)=0$
if $n\equiv 3\ \mathrm{mod}\; 4$, and $L^s_{\mathrm{odd}}(\mathbb Z_d)=0$
provided $d$ is odd. Thus $\bar f$ is homotopic to
a diffeomorphism except possibly when $n\equiv 1\ \mathrm{mod}\; 4$ and
$d$ is even, in which case $L^s_{n+3}(\mathbb Z_d\to 1)=\mathbb Z_2$
because $L^s_{n+2}(\mathbb Z_d)=\mathbb Z_2$ and $L^s_{n+2}(1)=0$.
\end{proof}

The next result is needed to construct further examples of complete,
noncompact, simply connected manifolds with metrics of nonnegative
sectional curvature.

\begin{lem}
\label{lem: curv for circle bundles}
Let $P\to B$ be a principal circle bundle whose total space $P$ is
$2$-connected and has an $S^1$-invariant metric
of nonnegative sectional curvature. Then $H^2(B)\cong\mathbb Z$ and
the total spaces of all complex line bundles over $B$
support complete metrics of nonnegative sectional curvature such
that the zero sections are souls.
\end{lem}

\begin{proof}
By the homotopy sequence of the fibration $P\to B$ we see that
$B$ is simply connected and $\pi_2(B)\cong\mathbb Z$, and by the
Hurewicz and Universal Coefficient Theorems we have $H^2(B)=\mathbb Z$.
Note that every complex line bundle $\omega$ over $B$ can be written
as $P\times_\rho\mathbb C$ for some representation
$\rho\co S^1\to U(1)$ (because the vanishing of
$H^2(P)$ implies the triviality of the pullback of $\omega$
via the projection $P\to B$, and
$\rho$ comes from the $S^1$-action on the
$\mathbb C$-factor of $P\times\mathbb C$).
The product metric on $P\times\mathbb C$ has nonnegative sectional curvature,
and it descends to a complete  metric  on $P\times_\rho\mathbb C$
of nonnegative sectional curvature with soul $P\times_\rho\{0\}$, which can be
identified with $B$.
\end{proof}

\begin{cor}
\label{cor: metric on line bundles}
Let $S$ be an Eschenburg space, Witten manifold, or a product
$\Sigma^7(d)\times\cptwom$ with $m\ge 1$. Then the total
space of every complex line bundle over $S$ admits a complete
metric of nonnegative sectional curvature with soul equal to
the zero section.
\end{cor}
\begin{proof}
Recall that Eschenburg spaces and Witten manifolds appear as
quotients of $SU(3)$ and
$S^5\times S^3$, respectively, by free circle actions
which preserve metrics of nonnegative sectional curvature.
Similarly, $\Sigma^7(d)\times\cptwom$ with $m\ge 1$
is the quotient of $\Sigma^7(d)\times S^{4m+1}$
by the free circle action that is trivial on the first factor
and standard on the second one; this circle action is isometric
with respect to the product
of a metric of nonnegative sectional curvature constructed in~\cite{GroZil}
and the standard metric on $S^{4k+1}$.
Therefore these spaces satisfy the assumptions of
Lemma~\ref{lem: curv for circle bundles}, and accordingly they support
metrics with the appropriate properties.
\end{proof}

\begin{proof}[Proof of Theorem~\ref{intro-thm: nondiffe souls}]
Fix homeomorphic, nondiffeomorphic manifolds $S$, $S^\prime$
that are products $\Sigma^7(d)\times\cptwom$ with $m\ge 1$,
Eschenburg spaces, or Witten manifolds. The existence of such pairs is ensured by
Theorem~\ref{intro-thm: CPq cancells} in the first case
and by results of~\cite{KS-wit, CEZ} in the remaining cases.
We claim that in each case $S^\prime$ is diffeomorphic to the connected sum of $S$
with a homotopy sphere that bounds a parallelizable manifold.
For products $\Sigma^7(d)\times\cptwom$ this easily follows as
in the proof of Corollary~\ref{cor: diffeo criterionfor sigma7(d)}.
For Eschenburg spaces or Witten manifolds this is implied
by smoothing theory and the fact that their third cohomology groups
with $\mathbb Z_2$-coefficients vanish; the crucial point is that if
the manifold $M_0$ is obtained by removing the interior of a closed
coordinate disk
from a closed $7$-manifold $M$ with $H^3(M;\mathbb Z_2)=0$, then
$M_0$ has a unique smooth structure because
$H^3(M_0;\mathbb Z_2)=0$ and the Kirby-Siebenmann map
$Top/O\to K(\mathbb Z_2,3)$ is $7$-connected
(compare \cite[p. 123]{HiM} and \cite[Essay V, Sections 4 and 5]{KiSi}).

Recall that every
element of $H^2(S)\cong\mathbb Z$ is the first Chern class
of a unique complex line bundle over $S$.
Lemma~\ref{lem: curv for circle bundles} shows that
$H^2(S)\cong\mathbb Z$, so there exists a nontrivial
complex line bundle over $S$. By Theorem~\ref{thm: diffeo total spaces},
this line bundle and its pullback
via the standard homeomorphism $S^\prime\to S$ have diffeomorphic
total spaces, and by Corollary~\ref{cor: metric on line bundles}
these total spaces admit complete
metrics of nonnegative sectional curvature with souls diffeomorphic to $S$
and $S^\prime$.
\end{proof}

\begin{proof}[Proof of Theorems~\ref{intro-thm: classification for s7xcp2}
and~\ref{intro-thm: classification for witten}]
By Theorem~\ref{intro-thm: CP2 cancells}
the manifolds $\Sigma^7(i)\times\cptwo$ with $i\in\{0,1,2\}$
are pairwise nondiffeomorphic. As in the proof of
Theorem~\ref{intro-thm: nondiffe souls} we conclude that
the total space $N$ of a nontrivial vector bundle over
$S^7\times\cptwo$ admits three complete metrics of nonnegative sectional curvature
with souls $S_i$ diffeomorphic to $\Sigma^7(i)\times\cptwo$, where
$i\in\{0,1,2\}$.

Suppose that $S$ is a soul of an arbitrary metric of nonnegative sectional curvature
on $N$. There is a canonical homotopy equivalence
$f_i\co S\to S_i$
given by the inclusion  $S\to N$ followed by the normal
bundle projection $N\to S_i$. It was shown
in~\cite[Corollary 4.2, Proposition 4.4]{BKS1} that
$f_i$ has trivial normal invariant in $[S, F/O]$.
Therefore by surgery theory $S$ is diffeomorphic to
$S_i\#\Sigma^{11}(d)$ for some $d$,
and after composing with this diffeomorphism
$f_i$ becomes homotopic to the connected sum of $\mathbf{id}(S_i)$
with the standard homeomorphism $\Sigma^{11}(d)\to S^{11}$.
Setting $i=0$ we conclude from Fact~\ref{fact: browder}
that $S$ is diffeomorphic to $S_d$ where
by Theorem~\ref{intro-thm: CP2 cancells} we may choose $d$
in $\{0,1,2\}$. Let $\phi\co S\to S_d$ be a diffeomorphism.

By~\cite[Corollary 4.2]{BKS1} the Euler classes of the normal bundles
of $S$, $S_d$ are preserved by $f_{d}$, and since
$H^2(S)\cong H^2(S_d)\cong\mathbb Z$, their Euler classes are also
preserved by $\phi$ up to sign.
So after changing orientation if needed, we may conclude that $\phi$
preserves the Euler classes of the normal bundles $\nu$, $\nu_d$
of $S$, $S_d$ in $N$, and hence the normal bundles themselves, for
the Euler class determines an oriented 2-dimensional vector bundle
up to isomorphism. In other words, the pullback bundle
$\phi^\#\nu_d$ is isomorphic to $\nu$, and in particular,
the pairs $(N, S)$ and $(N, S_d)$ are diffeomorphic.

The same proof works for the Witten manifold $M_{k,l}$
where $k,l$ are as in the assumptions of
Theorem~\ref{intro-thm: classification for witten}.
The only difference is that Theorem~\ref{intro-thm: CP2 cancells}
must be replaced by
the remark after~\cite[Corollary C]{KS-wit}; namely, under
our assumptions on $k$ and $l$ every smooth manifold that is
homeomorphic to $M_{k,l}$ must be a Witten manifold.
\end{proof}

\section{Manifolds tangentially homotopy equivalent to $S^7\times\cptwo$}
\label{sec: tang hom eq}

As a by-product of our methods we prove the following result,
which is not used elsewhere in the paper.

\begin{thm}\label{thm: tang hom eq S7xCP2}
If the closed manifold $M^{11}$ is tangentially homotopy
equivalent to $S^7\times\cptwo$ and $d$ is an odd integer, then
$M^{11}$ is diffeomorphic to exactly one of the manifolds $S^7\times\cptwo$,
$\Sigma^7(d)\times\cptwo$, or $\Sigma^7(2d)\times\cptwo$.
\end{thm}

The proof of Theorem~\ref{thm: tang hom eq S7xCP2}
relies heavily on the known structure of the
stable homotopy groups $\pi_*^{\bf S}$ in relatively low dimensions,
and thus there is no reason to expect a similar conclusion
if $\cptwo$ is replaced by $\cptwom$ for most (in fact, almost all) choices
of $m\geq 2$.

\begin{proof}
By Theorem~\ref{thm: index 4 in inertia}
it suffices to show that $M^{11}$ is diffeomorphic to
$\Sigma^7(d)\times\cptwo$ for some $d$.
The key point is to understand the normal
invariant ${\mathfrak q}(h)$ of an arbitrary tangential homotopy
equivalence $h:M^{11}\to S^7\times \cptwo$. Since
$h$ is a tangential homotopy equivalence, it follows that the
normal invariant ${\mathfrak q}(h)$ is the image of
some class $\theta\in
[S^7\times\cptwo, F]$ (compare the discussion preceding
Proposition~\ref{prop: tangentials}).

The exact cofiber sequence for the quotient map
$S^7\times \cptwo\to S^7\wedge \cptwo=\Sigma^7\cptwo$
\[
[S^7\cptwo, F]\to
[S^7\times\cptwo, F]\to\pi_7(F)~\oplus~[\cptwo, F]
\]
maps into the similar exact cofiber sequence for
$[\Sigma^7\times\cptwo, F/O]$.
Both sequences split via precomposing with projections
onto the $S^7$ and $\cptwo$ factors, and one has
a commutative diagram in which each of the three
components of $[S^7\times\cptwo, F]$
is mapped into the corresponding component of $[S^7\times\cptwo, F/O]$.

Since $\cptwo$ is the mapping cone of the Hopf map
$\eta_2\co S^3\to S^2$, we know that $[\cptwo, F]$
fits into the following exact cofiber sequence
for the map $\cptwo\to \cptwo/\cpone=S^{4}$.
\[
\pi_4^{\bf S}~\to~[\cptwo,F]~\to~\pi_2^{\bf S}~
\stackrel{\eta^\ast}{\to}~\pi_3^{\bf S},
\]
and hence
$[\cptwo,F]=0$ because $\pi_4^{\bf S}=0$,
and composition with $\eta$ induces a monomorphism from
$\pi_2^{\bf S}$ to $\pi_3^{\bf S}$
as it sends $\eta^2$ to $\eta^3 = 4\nu$,
which has order $2$ (see~\cite[Chapter XIV]{Tod}).
Since $\pi_7(F/O)=0$, it follows that
${\mathfrak q}(h)$ lies in $[\Sigma^7\cptwo, F/O]$;
by the previous paragraph, this implies that
${\mathfrak q}(h)$ is the image of a class in
$[\Sigma^7\cptwo, F]$ which we shall denote by
$\theta'$.

The rows of the commutative diagram below
are exact cofiber sequences for
the map $\Sigma^7\cptwo\to \Sigma^7(\cptwo/\cpone)=S^{11}$,
and columns are portions of
the exact homotopy sequence of the fibration $p\co F\to F/O$.
\[
\xymatrix{
0=\pi_{11}(F/O)\ar[r]&
[\Sigma^7\cptwo, F/O]\ar[r]&\pi_9(F/O)\ar[r]&\pi_{10}(F/O)\\
\pi_{11}(F)\ar[r]\ar[u]&
[\Sigma^7\cptwo, F]\ar[r]\ar[u]&\pi_9(F)\ar[r]^{\eta^\ast\ \ }\ar[u]^{p^\ast}&
\pi_{10}(F)\ar[u]\\
& &\pi_9(O)\ar[r]\ar[u]^{\text{1-1}}&
\pi_{10}(O)=0\ar[u]
}
\]
Since $\pi_{11}(F/O)=0$, we identify
$[\Sigma^7\cptwo, F/O]$ with the kernel of the map
$\pi_9(F/O)\to\pi_{10}(F/O)$
so ${\mathfrak q}(h)$ gets identified with
${\mathfrak q}(h)\vert {\Sigma^7\cpone}\in\pi_9(F/O)$.
Note that $p^\ast$ maps $\theta'\vert {\Sigma^7\cpone}$ to
${\mathfrak q}(h)\vert {\Sigma^7\cpone}$,
and by exactness
$\theta'\vert {\Sigma^7\cpone}\in \mathrm{ker}(\eta^\ast)$.
Thus the normal invariant of every tangential homotopy equivalence
$h$ lies in $p^\ast (\mathrm{ker}(\eta^\ast))$.

By~\cite[Chapter XIV]{Tod},
$\pi_9(F)=\mathbb Z_2\oplus\mathbb Z_2\oplus\mathbb Z_2$
with factors generated by $\nu^3$, $\mu$ and $\eta\tinycirc\e$,
where $\eta^*$ acts by precomposing with $\eta$, which
stably and up to sign amounts to postcomposing with
$\eta$~\cite[Proposition 3.1]{Tod}. Using~\cite[Theorem 14.1]{Tod}
we see that $\eta^*$ maps $\nu^3$ and $\eta\tinycirc\e$
to zero, while $\eta^\ast(\mu)=\mu\tinycirc\eta$ is nonzero.
The $J$-homomorphism $\mathbb Z_2=\pi_{9}(O)\to\pi_{9}(F)$
is one-to-one, and its image lies in $\mathrm{ker}(\eta^\ast )$
because ${\eta^\ast}{\tinycirc}{J}$ factors through $\pi_{10}(O)=0$.
Thus the
subgroup $p^\ast (\mathrm{ker}(\eta^\ast ))$ has order $2$.

As we mentioned in
Remark~\ref{rmk: existence of self-eq with stably nontrivial norm inv}
there exists a tangential
homotopy self-equivalence $f$ of $S^7\times\cptwo$
such that ${\mathfrak q}(f)\vert {\Sigma^7\cpone}$ is nonzero.
Since both ${\mathfrak q}(h)$ and ${\mathfrak q}(f)$ lie
in an order two subgroup, either ${\mathfrak q}(h)$
is trivial or it is equal to ${\mathfrak q}(f)$.
In the former case ${\bf id}(S^7\times\cptwo)$ and $h$
are in the same $bP_{12}$-orbit, and the same is true
in the latter case for the classes of $f$ and $h$.
Thus in either case $M$ is diffeomorphic to
$\Sigma(d)\,\#\,(S^7\times\cptwo)$ for some $d$, as claimed.
\end{proof}

\begin{rmk}
By contrast, every closed manifold $M$ that is tangentially
homotopy equivalent to $S^3\times\cptwo$ must be diffeomorphic
to $S^3\times\cptwo$. Indeed, by~\cite[Corollary 4.2]{MasSch}
the connected sum of $S^3\times\cptwo$ with $\Sigma^7(1)$ is
diffeomorphic to $S^3\times\cptwo$, so it suffices to show
that the tangential homotopy equivalence $f\co M\to S^3\times\cptwo$
has trivial normal invariant. Now $[\Sigma^3\cptwo, F/O]=0$ because
it fits into the exact sequence between the zero groups
$\pi_7(F/O)$ and $\pi_5(F/O)$, and moreover, $\pi_3(F/O)=0$,
so the restriction $[S^3\times\cptwo, F/O]\to [\cptwo, F/O]$ is injective.
The claim now follows because ${\mathfrak q}(f)$ comes from
$[S^3\times\cptwo, F]$, and the composition
$[S^3\times\cptwo, F]\to [S^3\times\cptwo, F/O]\to [\cptwo, F/O]$
factors through $[\cptwo, F]=0$.
\end{rmk}

\end{document}